\UseRawInputEncoding
\documentclass[a4paper,12pt,number,preprint]{article}
\usepackage{amsmath,amssymb,amsthm}
\usepackage{cite}
\usepackage{hyperref}

\allowdisplaybreaks[4]

\newtheorem{thm}{Theorem}[section]

\newtheorem{lem}{Lemma}[section]
\newtheorem{prop}{Proposition}[section]
\newtheorem{defn}{Definition}[section]
\newtheorem{rem}{Remark}[section]

\def\d{\partial}
\def\ddj{\dot \Delta_j}

\newcommand\R{\mathbb{R}}

\newcommand\Z{\mathbb{Z}}

\newcommand{\myref}[1]{}

\renewcommand{\div}{\mbox{\rm div}\;\!}

\def\cQ{{\mathcal Q}}

\begin{document}

\title{Decay of higher order derivatives for $L^p$ solutions to the compressible fluid model of Korteweg type}

\author{Zihao Song,\, Jiang Xu}

\date{}
\maketitle

\begin{abstract}
We present a new derivation for the optimal decay of \textit{arbitrary} higher order derivatives for $L^p$ solutions to the compressible fluid model of Korteweg type. This approach, based on Gevrey estimates, is to establish uniform bounds on the growth of the radius of analyticity of the solution in negative Besov norms. For that end, the maximal regularity property involving Gevrey multiplier of heat kernel and non standard product Besov estimates are well developed. Our approach is partly inspired by Oliver-Titi's work and is applicable to a wide range of dissipative systems.
\end{abstract}
\vskip 0.3 true cm

\vskip 0.3 true cm


\section{Introduction}\setcounter{equation}{0}
We are concerned with the compressible Navier-Stokes-Korteweg system, which is used to study the dynamics of a two-phases liquid-vapor mixture in the Diffuse Interface (DI) approach (see e.g., \cite{CR}). The theory formulation was originated from Van der Waals \cite{V}, later generalized by Korteweg \cite{K}. The basic idea is to add to the classical compressible fluids equations a \textit{capillary} term, which penalizes high variations of the density of phase transitions. The mathematical derivation of the corresponding equations is performed by Dunn and Serrin in \cite{DS}. The barotropic model reads
\begin{equation}
\left\{
\begin{array}{l}\partial_{t}\rho+\mathrm{div}(\rho u)=0,\\ [1mm]
 \partial_{t}(\rho u)+\mathrm{div}(\rho u \otimes u)+\nabla P=\mathcal{A}u+\mathrm{div}K.\\[1mm]
 \end{array} \right.\label{1.1}
\end{equation}
Here, $\rho=\rho(t,x)\in \mathbb{R}_{+}$ and $u=u(t,x)\in \mathbb{R}^{d}(d\geq2)$ are the unknown functions on $[0,+\infty)\times \mathbb{R}^{d}$, which stand for the density and velocity field of a fluid, respectively. The pressure $P=P(\rho)$ is a suitable smooth function of $\rho$. The diffusion operator
$\mathcal{A}u$ is given by $\mathcal{A}u=\mathrm{div}(2\mu D(u))+\nabla(\lambda\mathrm{div}u)$, where the Lam\'{e} coefficients $\lambda$ and $\mu$ (the \textit{bulk} and \textit{shear viscosities}) are density-dependent functions satisfying
$$\lambda>0, \nu\triangleq\lambda+2\mu>0.$$
The positive condition ensures the uniform ellipticity of $\mathcal{A}u$. The symmetric matrix $D(u)\triangleq\frac{1}{2}(\nabla u+{}^\top\!\nabla u)(\top$ transpose) stands for  the  deformation tensor, and $\nabla$ and $\div$ are the gradient and divergence operators with respect to the space variable. The Korteweg tensor is presented by (see \cite{BDDJ})
$$\mathrm{div}K=\nabla(\rho\kappa(\rho)\Delta\rho+\frac{1}{2}(\kappa(\rho)+\rho\kappa'(\rho))|\nabla\rho|^{2})-\mathrm{div}(\kappa(\rho)\nabla\rho\otimes\nabla\rho),$$
where $\nabla\rho\otimes\nabla\rho$ stands for the tensor product $(\partial_j\rho\partial_k\rho)_{jk}$ and the capillarity coefficient $\kappa>0$
may depend on $\rho$ in general. System \eqref{1.1} is supplemented with the initial conditions
\begin{equation}\label{initial}
\left(\rho,u\right)|_{t=0}=\left(\varrho _{0}(x),u_{0}(x)\right)
\end{equation}
We investigate the solution $(\rho,u)$ to the Cauchy problem \eqref{1.1}-\eqref{initial} fulfilling $\rho_0\rightarrow\rho^{\ast}$ and $u_0\rightarrow 0$ at infinity, where $\rho^{\ast}$ is a given positive constant.

Following from the spectral analysis in \cite{SX}, it is shown that the linearized system of \eqref{1.1} with zero sound speed $P'(\rho^{\ast})=0$ is purely parabolic in case that $\bar{\nu}^2\geq4\bar{\kappa}$ (see the scaled physical coefficients below). In the regime of ``small solutions", the dissipative mechanism is parallel to that of the incompressible Navier-Stokes equations. Therefore, let us first recall classical time-decay efforts for incompressible fluids.

On a periodic domain, the generic decay is exponential with a rate given by the lowest eigenvalue of the Stokes operator, see the work by Foias and Saut \cite{FS}. On the whole space $\mathbb{R}^d$, due to the lack of a positive lowest eigenvalue of the Stokes operator, the decay rates are generally algebraic rather than exponential in time. See for example previous works by Kato \cite{Ka}, Schonbek \cite{S1,S2}, Kajikiya \& Miyakawa\cite{KM}, Wiegner \cite{W} and other mathematicians who have established upper bounds on the decay rate of solutions:
\begin{eqnarray}\label{heat}\|u(t)\|_{L^2(\mathbb{R}^d)}\lesssim (1+t)^{-d/4},\end{eqnarray}
provided that initial data $u_0\in L^2\cap L^1$. A lower bound with the same order of decay can be established by Schonbek \cite{S3} if initial data lie outsider a set of functions of radially equidistributed energy. Schonbek \cite{S4} also showed the decay of higher-order norms of the solutions of Navier-Stokes equations in two dimensions. Oliver \& Titi \cite{OT} established, by applying Gevrey analyticity which was addressed by Fioas and Temam \cite{FT1,FT2}, the explicit bounds on the growth of the radius of analyticity of the solution in time, which implies the optimal decay estimates of \textit{arbitrary} higher-order derivatives in Sobolev framework. The key inequality in their paper is as follows (see also Lemma 9 in \cite{OT}):
\begin{equation}\label{R-E4}
\|\Lambda^{q}u\|_{L^2}^{2}\leq c(p,q)\tau^{p-2q}\|u\|_{L^2}\|\Lambda^{p}e^{\tau \Lambda}u\|_{L^2}\end{equation}
for $0\leq p \leq2q$ and $\tau>0$, where the pseudo-differential operator $\Lambda$ is defined by $\Lambda\triangleq \sqrt{-\Delta}$. The observation is that the radius of uniform analyticity increases like $\sqrt {t}$ as $t\rightarrow\infty$ as for solutions of the heat equation.

Inspired by \eqref{R-E4}, we shall develop an idea in Besov framework as follows
\begin{eqnarray}\label{key!}
\|\Lambda^{l}u\|_{\dot{B}^{0}_{2,1}}
\lesssim t^{-\frac{l}{2}-\frac{\sigma}{2}}\|e^{\sqrt {t} \Lambda_{1}}u\|_{\dot{B}^{-\sigma}_{2,\infty}}\,\,\,\,\mathrm{for}\,\,\,\,l>-\sigma.
\end{eqnarray}
That is, our approach consists of two steps:
\begin{itemize}
\item Transform higher-order derivative estimates into lower-order analytic estimates;
\item Establish uniform bounds on the growth of the radius of analyticity in negative Besov norms
$$\|e^{\sqrt {t} \Lambda_{1}}v\|_{\dot{B}^{-\sigma}_{2,\infty}}
\leq C\,\,\,\,\mathrm{for}\,\,\,\,t>0,$$
\end{itemize}
where $\Lambda_1$  stands for the Fourier multiplier with
symbol $|\xi|_{1}=\sum_{i=1}^{d}|\xi_{i}|.$\footnote{For technical reasons, it is more
convenient to use the  $\ell^1(\R^d)$ norm instead of the usual $\ell^2(\R^d)$ norm associated with $\Lambda$, see \cite{BBT,CDX}.} Furthermore, choosing a suitable regularity (for instance, $\sigma=d/2$) enables us to get the optimal decay estimates like \eqref{heat} for solution and its high order derivatives in the $L^2$ norm. It should point out here that the approach can be applied to a wide range of dissipative parabolic systems of Gevrey analyticity. In the following, we consider the large-time behavior of $L^p$ solutions to the Cauchy problem \eqref{1.1}-\eqref{initial} with zero sound speed.

\subsection{Previous works}
In the past decades, the Korteweg system \eqref{1.1} has been received more and more attention due to the physical importance.
The existence of strong solutions for \eqref{1.1} was known since the works by Hattori and Li \cite{HL,HL1}. Global solutions were obtained
only for initial data close enough to a stable equilibrium $(\rho^{\ast},0)$ with convex pressure profiles. Bresch, Desjardins and Lin \cite{DDL}
established the global existence of weak solutions in a periodic or strip domain. However, the uniqueness problem of weak solutions has not been solved.
A natural way of dealing with the uniqueness is to find a functional setting as large as possible in which the existence and uniqueness hold. This idea
is closely linked with the concept of scaling invariance space (critical Besov spaces), which has been successfully employed by Fujita-Kato \cite{FK}, Cannone \cite{Can} and Chemin \cite{C} for incompressible Navier-Stokes equations. Later, Bae-Biswas-Tadmor \cite{BBT} established the analyticity and decay with small data. Danchin \cite{D} has developed the idea of scaling invariance in compressible Navier-Stokes equations. Note that the fact that \eqref{1.1} is invariant by the transformation
\begin{eqnarray*}
\rho(t,x)\leadsto \rho(\ell ^2t,\ell x),\quad
u(t,x)\leadsto \ell u(\ell^2t,\ell x), \ \ \ell>0
\end{eqnarray*}
up to a change of the pressure term  $P$ into $\ell^2P$, Danchin and Desjardins \cite{DD} investigated the global well-posedness of \eqref{1.1} in critical Besov spaces for initial data close enough to stable equilibrium $(\rho^{\ast},0)$ with the stability assumption $P'(\rho^{\ast})>0$. Antonelli and Spirito \cite{AS} established the global existence of finite energy weak solutions for large initial data, where vacuum regions are allowed in the definition of weak solutions. Charve, Danchin and the second author \cite{CDX} investigated the global existence and Gevrey analyticity of \eqref{1.1} in more general critical $L^p$ framework, which exhibits Gevrey analyticity for a model of compressible fluids. Murata and Shibata \cite{MS} addressed a totally different statement on the global existence of strong solutions to \eqref{1.1} in Besov spaces, where the
maximal $L^p$-$L^q$ regularity was mainly employed. For the large-time behavior of solutions, Tan and Wang \cite{TW} deduced various optimal time-decay rates of smooth solutions and their spatial derivatives. Chikami and Kobayashi \cite{CK} studied the optimal time-decay estimates in the
$L^2$ critical Besov spaces. Recently, Kawashima, Shibata and the second author \cite{KSX1} investigated the dissipation effect of Korteweg tensor with the density-dependent capillarity and developed the $L^p$ energy methods (independent of spectral analysis), which leads to the optimal time-decay estimates of strong solutions.

We would like to mention that all above results are dedicated to the stable case $P'(\rho^{\ast})>0$. It is well-known from \cite{DS} that the Korteweg system \eqref{1.1} was deduced by using Van der Waals potential, where the pressure law is not necessary increasing. Therefore, it is interesting to investigate more physical cases $P'(\rho^{\ast})=0$ (zero sound speed) and $P'(\rho^{\ast})<0$. In those cases, the pressure term couldn't provide any dissipation. Danchin and Desjardins \cite{DD} focused on the critical case $P'(\rho^{\ast})=0$ and gave a spectral study of the corresponding linearized system, which indicates the parabolic smoothing is available in all frequency spaces. Kotschote \cite{K2} considered the initial-boundary value problem in bounded domain and proved the local existence and uniqueness of strong solutions in maximal $L^p$-regularity class. Chikami and Kobayashi \cite{CK} established the global well-posedness and decay of strong solutions in $L^2$ critical Besov space for the case of $P'(\rho^{*})=0$. For the critical case, Huang, Hong and Shi \cite{HHS} also proved the local-in-time existence of smooth solutions to \eqref{1.1}-\eqref{initial}. The global-in-time existence of smooth solutions was also established in periodic domain.

\subsection{Reformulation and main results}
In this paper, we are interested in studying the global analyticity and decay in the case of zero sound speed. Denote by $a=(\rho-\rho^{*})/\rho^{*}$ the density fluctuation and by $m=\rho u/\rho^{*}$
the scaled momentum. Also, we introduce the scaled viscosity coefficients $\bar{\mu}=\frac{\mu(\rho^{*})}{\rho^{*}}$, $\bar{\lambda}=\frac{\lambda(\rho^{*})}{\rho^{*}}$
and the scaled capillarity coefficient $\bar{\kappa}=\kappa(\rho^{*})\rho^{*}$. A simple calculation leads to the following perturbation form
\begin{equation}
\left\{
\begin{array}{l}\partial_{t}a+\mathrm{div}m=0,\\ [1mm]
 \partial_{t}m-\bar{\mathcal{A}}m-\bar{\kappa}\nabla\Delta a= g(a,m),\\[1mm]
(a,m)|_{t=0}=(a_{0},m_{0}),\\[1mm]
 \end{array} \right.\label{linearized}
\end{equation}
where
$$a_{0}=(\rho_{0}-\rho^{*})/\rho^{*}, \quad  \ m_{0}=\rho_{0} m_{0}/\rho^{*}, \quad \bar{\mathcal{A}}m\triangleq \bar{\mu}\Delta m+(\bar{\mu}+\bar{\lambda})\nabla\mathrm{div}m $$
and the nonlinear term $g(a,m)$ will be explicitly given in Section 3. For convenience of readers, we would like to summarize a recent result for the global wellposedness and analyticity in a class of hybrid Besov spaces that was introduced by Danchin \cite{D} and generalized by Haspot \cite{H2}.

Let $\dot{\Delta}_{j}$ and $\dot S_j$ be the Fourier cut-off operators (see \cite{BCD}). The Littlewood-Paley decomposition of a general tempered distribution $f\in \mathcal{S}'$ reads
\begin{equation}\label{eq:decompo}
f=\sum_{j\in\Z}\ddj f.
\end{equation}
As it holds only
modulo polynomials, it is convenient to consider ${S}'_{0}$ be the subspace of those tempered distributions such that
\begin{equation}\label{eq:Sh}
\lim_{j\rightarrow-\infty}\|\dot S_jf\|_{L^\infty}=0.
\end{equation}
Indeed, if \eqref{eq:Sh} is fulfilled then \eqref{eq:decompo} holds in $\mathcal{S}'$. Let us now turn to the definition of hybrid Besov spaces.

\begin{defn}\label{defn hybrid}
Let $s,t\in \mathbb{R}$, $p,q,r_{1},r_{2}\in [1,\infty]$. We denote $\dot{B}^{s,t}_{(p,r_{1}),(q,r_{2})}$ by the space of functions $f\in\mathcal{S}'_{0}$ equipped with norm:
\begin{eqnarray*}\|f\|_{\dot{B}^{s,t}_{(p,r_{1}),(q,r_{2})}}=\Big\{\sum_{j\geq j_{0}}2^{sjr_{1}}\|\dot{\Delta}_{j}f\|^{r_{1}}_{L^{p}}\Big\}^{\frac{1}{r_{1}}}+\Big\{\sum_{j<
j_{0}}2^{tjr_{2}}\|\dot{\Delta}_{j}f\|^{r_{2}}_{L^{q}}\Big\}^{\frac{1}{r_{2}}},
 \end{eqnarray*}
for some integer $j_{0}$. For convenience, we write $\|f\|_{\dot{B}^{s,t}_{(p,r_{1}),(q,r_{2})}}\triangleq\|f\|^{h}_{{{\dot B}^{s}}_{p,r_{1}}}+\|f\|^{\ell}_{{{\dot B}^{t}}_{q,r_{2}}}.$
\end{defn}

Moreover, one can define the hybrid Chemin-Lerner spaces
$\tilde{L}^{\rho_{1},\rho_{2}}_{T}(\dot{B}^{s,t}_{(p,r_{1}),(q,r_{2})})$ with norm:
 \begin{eqnarray*}\|f\|_{\tilde{L}^{\rho_{1},\rho_{2}}_{T}(\dot{B}^{s,t}_{(p,r_{1}),(q,r_{2})})}=\Big\{2^{sj}\|\dot{\Delta}_{j}f\|_{L^{\rho_{1}}_{T}L^{p}}\Big\}_{l^{r_{1}}_{j\geq
 j_{0}}}+\Big\{2^{tj}\|\dot{\Delta}_{j}f\|_{L^{\rho_{2}}_{T}L^{q}}\Big\}_{l^{r_{2}}_{j< j_{0}}}
 \end{eqnarray*}
for $T>0$. Similarly, we agree that $\|f\|_{\tilde{L}^{\rho_{1},\rho_{2}}_{T}(\dot{B}^{s,t}_{(p,r_{1}),(q,r_{2})})}\triangleq\|f\|^{h}_{\tilde{L}^{\rho_{1}}_{T}(\dot{B}^{s}_{p,r_{1}})}+
\|f\|^{\ell}_{\tilde{L}^{\rho_{2}}_{T}(\dot{B}^{t}_{q,r_{2}})}$. Also,
$\tilde{L}^{\rho,\rho}_{T}(\dot{B}^{s,t}_{(p,r_{1}),(q,r_{2})})=\tilde{L}^{\rho}_{T}(\dot{B}^{s,t}_{(p,r_{1}),(q,r_{2})})$. For notational simplicity, index $T$ is omitted if $T=+\infty$, changing $[0,T]$ to $[0,+\infty)$ in the definition above.

The global wellposedness and analyticity of strong solutions to the Cauchy problem \eqref{1.1}-\eqref{initial} are stated as follows, which has recently shown by \cite{SX}.
\begin{thm}\label{thm2.1}Let $\rho^{*}>0$ such that $P'(\rho^{*})=0$. Let $\bar{\nu}^2\geq4\bar{\kappa}$ and $1 \leq q\leq p\leq\min\{d, 2q\}$ with
\begin{eqnarray}\label{exist}\frac{1}{q}\leq \frac{1}{p}+\frac{1}{d}.\end{eqnarray}
There exists a positive $\eta>0$ depending on functions $\kappa, \lambda, \mu$ and $P$ and on $p, q$ and $d$ such that if $(a_{0},m_{0})\in \dot B^{\frac d{p}}_{p,1}\times\dot
B^{\frac d{p}-1}_{p,1}$, besides, $(a_0^\ell,u_0^\ell)\in{{\dot
    B}^{\frac{d}{q}-2}}_{q,\infty}\times{{\dot B}^{\frac{d}{q}-3}}_{q,\infty}$ satisfying
     \begin{eqnarray*}
\|(\nabla a_{0}, m_{0})\|_{\dot{B}^{\frac{d}{p}-1,\frac{d}{q}-3}_{(p,1),(q,\infty)}}\leq\eta,\end{eqnarray*}
     then \eqref{1.1}-\eqref{initial} admits a unique global-in-time solution $(a,m)$ in the space $E^{p,q}$ satisfying
\begin{equation}\label{bound}
\|(a,m)\|_{E^{p,q}_{T}}\lesssim \|(\nabla a_{0}, m_{0})\|_{\dot{B}^{\frac{d}{p}-1,\frac{d}{q}-3}_{(p,1)(q,\infty)}}
\end{equation}
for any $T>0$, where
\begin{equation*}
\|(a,m)\|_{E^{p,q}_{T}}\triangleq
\|(\nabla a,m)\|_{\tilde{L}^{\infty}_{T}(\dot{B}^{\frac{d}{p}-1,\frac{d}{q}-3}_{(p,1),(q,\infty)})}+\|(\nabla
a,m)\|_{\tilde{L}^{1}_{T}(\dot{B}^{\frac{d}{p}+1,\frac{d}{q}-1}_{(p,1),(q,\infty)})}.
\end{equation*}
Moreover, if those functions $\lambda, \mu, \kappa$ and $P$ are assumed to be real analytic near zero, then for $d\geq3$ and $1<q\leq p\leq\min\{d, 2q\}$ with $1/q\leq 1/p+1/d$, the solution $(a,m)$ fulfills $e^{\sqrt {c_{0}t}\Lambda_1}(a,m)\in E^{p,q}$,
where $c_{0}=c_{0}(d,\bar{\mu},\bar{\lambda},\bar{\kappa},\rho^{*})$ is some positive constant.
\end{thm}

Theorem \ref{thm2.1} indicates that the Korteweg system \eqref{1.1} is \textit{purely parabolic} in the case of $\bar{\nu}^2\geq4\bar{\kappa}$ and acoustic waves are not available. Consequently, the usual $L^2$ type bounds on the low frequencies of solutions are improved to the $L^p$ framework in contrast to the priori study of compressible Navier-Stokes equations (\cite{CD,CMZ,D,H2}) or compressible Navier-Stokes-Korteweg equations (\cite{CDX,CK,DD}). Also, System \eqref{1.1} exhibits the Gevrey analyticity, where the radius of uniform analyticity increases like $\sqrt {t}$ as $t\rightarrow\infty$. Stemming from the simple idea \eqref{key!}, we are able to develop a new decay framework involving Gevrey smoothing estimates, which leads to the optimal time-decay of the solution and its arbitrary higher-order derivatives. Now we state the main result of this paper.

\begin{thm}\label{thm2.3}
Let $(a,m)$ be the global solution to \eqref{linearized} addressed by Theorem \ref{thm2.1}. Suppose that the
real number $\sigma_1$ fulfills $2-\frac{d}{q}\leq\sigma_{1}<d-\frac{d}{q}$, if $1<p\leq2$ and $2-\frac{d}{q}\leq\sigma_{1}\leq \frac{2d}{p}-\frac{d}{q}$, if $p>2$.
If in addition initial norm $ \|(\nabla a_0,m_0)\|^{\ell}_{\dot{B}^{-\sigma_{1}-1}_{q,\infty}}$ is bounded,
then the solution $(a,u)$ satisfies the following decay estimates
\begin{eqnarray}\label{bound decay1}\|\Lambda^{l}a\|_{L^{r}}\leq C \langle t-t_0\rangle^{-\frac{\tilde{\sigma}_{1}}{2}-\frac{l}{2}}, \quad l>-\tilde{\sigma}_{1};
\end{eqnarray}
\begin{eqnarray}\label{bound decay2}
\|\Lambda^{l}m\|_{L^{r}}\leq C \langle t-t_0\rangle^{-\frac{\tilde{\sigma}_{1}}{2}-\frac{1}{2}-\frac{l}{2}}, \quad l>-\tilde{\sigma}_{1}-1,\end{eqnarray}
for all $t\geq t_0$ and $r\geq p$, where $t_{0}>0$ is some certain transient (sufficiently small) time, $\tilde{\sigma}_{1}\triangleq\sigma_{1}-\frac{d}{r}+\frac{d}{q}$   and $\langle t\rangle\triangleq\sqrt{1+t^2}$.
\end{thm}

\begin{rem}
Theorem \ref{thm2.3} exhibits the optimal long-time asymptotic description of solution and its arbitrary higher-order derivatives. Indeed, by choosing $\sigma_{1}=d-\frac{d}{q}-1$ ($L^1\hookrightarrow\dot{B}^{-\sigma_{1}-1}_{q,\infty}$), we arrive at
$$
\|\Lambda^{l}a\|_{L^{r}}\lesssim  \langle t\rangle^{-\frac{d}{2}(1-\frac{1}{r})+\frac{1}{2}-\frac{l}{2}}, \quad l>-\frac{d}{r'}+1;
\quad \|\Lambda^{l}m\|_{L^{r}}\lesssim  \langle t\rangle^{-\frac{d}{2}(1-\frac{1}{r})-\frac{l}{2}}, \quad l>-\frac{d}{r'}
$$
for $1<  r\leq \infty$ ($1/r+1/r'=1$). Owing to the absence of dissipation from the pressure in the case of zero sound speed, we see that the density decays
at a slower time-rate than the velocity. Those decay rates for $1<r<2$ are totally new, which provide a hint for long-time behaviors for compressible fluids.
\end{rem}

\begin{rem}
The proof of Theorem \ref{thm2.3} actually  presents a derivation for the upper bounds for the decay of higher-order derivatives of solutions. The key estimate lies in the uniform bounds on the growth of the radius of analyticity in Besov spaces of negative order and thus is different in comparison with the recent energy
methods, see \cite{DX,KSX1,XX,X}. The argument is of independent interest in the critical setting, which may be applicable to a wide range of dissipative
systems of ``\textit{regularity-gain type}" \cite{KSX2}.
\end{rem}

We would like to present some illustrations on the proof of Theorem \ref{thm2.3}. As analyzed in \eqref{key!},
firstly, we utilize Lemma \ref{lem5.2} to transform those estimates of higher order derivatives into lower order analyticity estimates, where
the radius of analyticity grows like $\sqrt{t}$ as $t\rightarrow\infty$. The next step is to establish uniform bounds on the growth of the radius
in the Besov space $\dot{B}^{-\sigma_{1}}_{q,\infty}$ (see \eqref{evoo}), which is actually the consequence of Proposition \ref{tol}. In order to prove
Proposition \ref{tol}, the $L^p$ energy method in terms of effective velocity that was first used in the critical framework by Haspot \cite{H2} is mainly employed. In the Korteweg case, the effective velocity is given by
 $$w\triangleq \cQ m+\alpha \nabla a\ \ \Big(\alpha=\frac{1}{2}(\bar{\nu}\pm\sqrt{\bar{\nu}^2-4\bar{\kappa}})\Big),$$
which allows to eliminate the coupling between $a$ and  $\mathcal{Q}m$ (the compressible part of momentum). The linear analysis
depends on the maximal regularity property involving Gevrey multiplier of heat kernel (see Lemma \ref{lem4.2}) and the nonlinear estimates resort to non standard product Besov estimates (see Lemma \ref{d1}). Finally, various Sobolev embeddings and interpolation enables us to establish the evolution of Gevrey regularity in Besov spaces restricted in low frequencies.

The rest of this paper unfolds as follows. Section 2 is devoted to Fourier multiplier lemmas, the maximal regularity involving Gevrey multiplier of heat kernel and nonlinear Gevrey estimates for product and composite.
Section 3 is the main part, which is dedicated to uniform bounds on the growth of the radius of analyticity. The proof of Theorem \ref{thm2.3} will be presented in the final Section 4.

\section{Preliminary}\setcounter{equation}{0}\label{sec:4}
Throughout the paper, $C>0$ stands for a harmless ``constant". For brevity, $f\lesssim g$ means that $f\leq Cg$. It will also be understood that $\|(f,g)\|_{X}=\|f\|_{X}+\|g\|_{X}$ for all $f,g\in X$. Moreover, for $ 1\leq p\leq \infty$, we denote by $L^{p}=L^{p}(\mathbb{R}^{d})$ the Lebesgue space on $\mathbb{R}^{d}$ with the norm $\|\cdot\|_{L^{p}}$. The following Fourier multiplier lemmas have been proved in \cite{BBT}.
\begin{lem}\label{lem3.1} Consider the operator $F:=e^{-(\sqrt{t-s}+\sqrt s-\sqrt t)\Lambda_1}$ for $0\leq s\leq t$. Then $F$ is either the identity operator or is an $L^1$ kernel whose
$L^1$ norm is bounded independent of $s,t$.
\end{lem}

\begin{lem}\label{lem3.2} The operator $F=e^{\frac{1}{2}a\Delta+{\sqrt a}\Lambda_1}$ is a Fourier multiplier which maps boundedly $L^p\to L^p$, $1< p<\infty$, and its operator norm is uniformly bounded with respect to $a\geq0$.
\end{lem}

Indeed, Lemma \ref{lem3.2} holds true for the operator $F=e^{a\Delta+m{\sqrt a}\Lambda_1}$ with $m\in \mathbb{R}$. Proving the Gevrey regularity of solutions will be based on continuity results for the family
$(\mathcal{B}_{t})_{t\geq 0}$ of bilinear operators, which are defined by
\begin{eqnarray*}
\mathcal{B}_{t}(f,g)(t,x)&=&e^{\sqrt {c_{0}t}\Lambda_1}(e^{-\sqrt {c_{0}t}\Lambda_1}fe^{-\sqrt {c_{0}t}\Lambda_1}g)(x)
\nonumber\\&=& \frac{1}{(2\pi)^{2d}}\int_{\mathbb{R}^d}\int_{\mathbb{R}^d}e^{ix\cdot(\xi+\eta)}e^{\sqrt {c_{0}t}(|\xi+\eta|_{1}-|\xi|_{1}-|\eta|_{1})}\widehat f(\xi)\widehat g(\eta)d\xi d\eta
\end{eqnarray*}
for some $c_0>0$. From \cite{BBT,L-cras}, one can introduce the following operators acting on functions depending
on one real variable:
$$K_{1}f=\frac{1}{2\pi}\int_{0}^{\infty}e^{ix\xi}\widehat f(\xi)d\xi,$$
$$K_{-1}f=\frac{1}{2\pi}\int_{-\infty}^{0}e^{ix\xi}\widehat f(\xi)d\xi,$$
and define $L_{a,1}$ and $L_{a,-1}$ as follows:
$$L_{a,1}f=f\quad \mbox{and}\quad L_{a,-1}f=\frac{1}{2\pi}\int_{\mathbb{R}^d}e^{ix\xi}e^{-2a|\xi|}\widehat f(\xi)d\xi.$$
Set
$$Z_{t,\alpha,\beta}=K_{\beta_{1}}L_{\sqrt{c_{0}t},\alpha_{1}\beta_{1}}\otimes...\otimes K_{\beta_{d}}L_{\sqrt{c_{0}t},\alpha_{d}\beta_{d}}
\quad \mbox{and}\quad K_{\alpha}=K_{\alpha_{1}}\otimes ...\otimes K_{\alpha_{d}}$$
for $t\geq 0$, $\alpha=(\alpha_{1}, ... ,\alpha_{d})$ and $\beta=(\beta_{1}, ...,\beta_{d})\in \lbrace-1,1\rbrace^d$.
Then it follows that
$$\mathcal{B}_{t}(f,g)=\sum_{(\alpha,\beta,\gamma)\in({\{-1,1\}^{d})}^3}K_{\alpha}(Z_{t,\alpha,\beta}fZ_{t,\alpha,\gamma}g).$$
It is not difficult to see that $K_{\alpha},Z_{t,\alpha, \beta}$ are linear combinations of smooth homogeneous of
degree zero Fourier multipliers, which are bounded on $L^p$ for $1<p<\infty$. Consequently,
\begin{lem}\label{lem5.3} (\cite{BBT})
For any $1<p,p_{1},p_{2}<\infty$ with $\frac{1}{p}=\frac{1}{p_{1}}+\frac{1}{p_{2}}$, we have for some constant $C$
independent of $t\geq0$,
$$\|\mathcal{B}_{t}(f,g)\|_{L^p}\leq C\|f\|_{L^{p_{1}}}\|g\|_{L^{p_{2}}}.$$
\end{lem}

Fixing some real number $c_0>0$, denote $V(t)\triangleq e^{\sqrt {c_{0}t}\Lambda_1}v$, $A(t)\triangleq e^{\sqrt {c_{0}t}\Lambda_1}a$
and $B(t)\triangleq e^{\sqrt {c_{0}t}\Lambda_1}b$ for $t\geq0$ (dependence on $t$  will be often omitted). We first present the analytic property for the heat equation.
\begin{lem}\label{lem4.2}
Let $T>0$, $\sigma\in \mathbb{R}, 1<p<\infty $ and $1\leq \rho_2, r\leq \infty.$ Let $v$ be the solution of heat equation
\begin{equation}\label{toy1.1}
\left\{
\begin{array}{l}\partial_{t}v-\mu\Delta v=F,\\ [1mm]
v|_{t=0}=v_{0}(x),\\[1mm]
 \end{array} \right.
\end{equation}
with $\mu>0$. Then, we have
\begin{eqnarray}\label{mmy}
\mu^{\frac{1}{\rho_{1}}}\|V\|_{\tilde{L}^{\rho_1}_{T}(\dot{B}^{\sigma+\frac{2}{\rho_1}}_{p,r})}
\lesssim\|v_{0}\|_{\dot{B}^{\sigma}_{p,r}}+\mu^{\frac{1}{\rho_{2}}-1}\|e^{\sqrt{c_{0}t}\Lambda_{1}}F\|_{\tilde{L}^{\rho_2}_{T}(\dot{B}^{\sigma-2+\frac{2}{\rho_2}}_{p,r})}
\end{eqnarray}
for all $\rho_1\in[\rho_2,\infty]$.
\end{lem}

\begin{proof} It follows from Duhamel's formula that
 \begin{equation}\label{duhamel}
v(t)=e^{\mu t \Delta}v_{0}+\int^{t}_{0}e^{\mu(t-\tau) \Delta}F(\tau)d\tau.
\end{equation}

Now applying Gevrey multiplier $e^{\sqrt{c_{0}t}\Lambda_{1}}$ on (\ref{duhamel}), we find $V$ satisfies the following equation
 \begin{equation*}
V(t)=e^{\sqrt{c_{0}t}\Lambda_1+\mu t \Delta}v_{0}+\int^{t}_{0}
e^{\sqrt{c_{0}(t-\tau)}\Lambda_1+\mu(t-\tau) \Delta}
e^{\sqrt{c_{0}}(-\sqrt{t-\tau}+\sqrt{t}-\sqrt{\tau})\Lambda_1}e^{\sqrt{c_{0}\tau}\Lambda_{1}}F(\tau)d\tau.
\end{equation*}
By taking $a=\frac{\mu t}{2} \,\,\text{with}\,\, m=\sqrt{2c_{0}/\mu}$, it follows from Lemmas \ref{lem3.1}-\ref{lem3.2} that
\begin{multline}
\|\dot{\Delta}_{j}V(t)\|_{L^{p}}\lesssim
\|\dot{\Delta}_{j}\big(e^{\frac{\mu}{2} t \Delta}v_{0}\big)\|_{L^{p}}+\int^{t}_{0}\big\|\dot{\Delta}_{j}\big(e^{\frac{\mu}{2}(t-\tau)\Delta}e^{\sqrt{c_{0}\tau}\Lambda_{1}}F(\tau)\big)\big\|_{L^{p}}d\tau\\
\lesssim
e^{-c\mu2^{2j}t}\|\dot{\Delta}_{j}v_{0}\|_{L^{p}}+\int^{t}_{0}e^{-c\mu2^{2j}(t-\tau)}\|\dot{\Delta}_{j}\big(e^{\sqrt{c_{0}\tau}\Lambda_{1}}F(\tau)\big)\|_{L^{p}}d\tau
\end{multline}
for some $c>0$.  Setting $\frac{1}{\rho_1}+1=\frac{1}{\theta}+\frac{1}{\rho_2}$. Consequently, performing $l^r$ norm after multiplying $2^{(\sigma +\frac{2}{\rho_1})j}$ and applying Young's inequality imply that
\begin{multline}
\|V\|_{\tilde{L}^{\rho_1}_{T}(\dot{B}^{\sigma+\frac{2}{\rho_1}}_{p,r})}
\lesssim\mu^{-\frac{1}{\rho_{1}}}\|v_{0}\|_{\dot{B}^{\sigma}_{p,r}}+
\big\{\|e^{-c\mu2^{2j}t}\|_{L^{\theta}_{T}}2^{(\sigma+\frac{2}{\rho_1})j}\|\dot{\Delta}_{j}\big(e^{\sqrt{c_{0}t}\Lambda_{1}}F(t)\big)\|_{L^{\rho_2}_{T}L^{p}}\big\}_{l^{r}}\\
\lesssim\mu^{-\frac{1}{\rho_{1}}}\|v_{0}\|_{\dot{B}^{\sigma}_{p,r}}+\mu^{\frac{1}{\rho_{2}}-\frac{1}{\rho_{1}}-1}\|e^{\sqrt{c_{0}t}\Lambda_{1}}F\|_{\tilde{L}^{\rho_2}_{T}(\dot{B}^{\sigma-2+\frac{2}{\rho_2}}_{p,r})},
\end{multline}
which is (\ref{mmy}) exactly.
\end{proof}

In order to match different Lebesgue indices at low frequencies and high frequencies, some product estimates involving Gevrey multiplier need to be developed, where different $q$ and $p$ are admissible.

\begin{lem}\label{d1}
Let $1< q, p<\infty$ such that $p\leq 2q$. Suppose that $s_{1}$ and $s_2$ fulfill with
$$s_{1}+s_{2}\geq\max d\Big(0, \frac{2}{p}-1\Big), \quad s_{1}\leq\min d\Big(\frac{1}{p}, \frac{2}{p}-\frac{1}{q}\Big), \quad s_{2}<\min d\Big(\frac{1}{p}, \frac{2}{p}-\frac{1}{q}\Big).$$
Then it holds that
\begin{eqnarray}\label{m1}\|e^{\sqrt {c_{0}t}\Lambda_{1}}(ab)\|_{\dot{B}^{s}_{q,\infty}}\lesssim\|A\|_{\dot{B}^{s_{1}}_{p,1}}\|B\|_{\dot{B}^{s_{2}}_{p,\infty}},\,\,\,\, \mathrm{with} \,\,\,\,s+\frac{2d}{p}-\frac{d}{q}=s_{1}+s_{2}.\end{eqnarray}
Additionally, if
$$s_{1}+s_{2}\geq\max d\Big(0, \frac{1}{p}+\frac{1}{q}-1\Big), \quad s_{1}\leq\frac{d}{p}, \quad s_{2}<\min d\Big(\frac{1}{p}, \frac{1}{q}\Big),$$
then it holds that
\begin{eqnarray}\label{m2}\|e^{\sqrt {c_{0}t}\Lambda_{1}}(ab)\|_{\dot{B}^{s}_{q,\infty}}\lesssim\|A\|_{\dot{B}^{s_{1}}_{p,1}}\|B\|_{\dot{B}^{s_{2}}_{q,\infty}},\,\,\,\, \mathrm{with} \,\,\,\,s+\frac{d}{p}=s_{1}+s_{2}.\end{eqnarray}
\end{lem}

\begin{proof}
We first focus on the proof of (\ref{m1}). The Bony's decomposition for the product of two tempered distribution (see for example \cite{BCD}) reads as
$$ab=T_{a}b+R(a,b)+T_{b}a,$$
where the paraproduct $T_{a}b$ is defined by
$$T_{a}b=\sum_{j\in \mathbb{Z}}\dot S_{j-1} a\dot{\Delta}_j b$$
and the remainder $R(a,b)$ is given by
$$R(a,b)=\sum_{j\in \mathbb{Z}}\dot{\Delta}_ja(\dot{\Delta}_{j-1}b+\dot{\Delta}_{j}b+\dot{\Delta}_{j+1}b).$$
Furthermore, employing the definition of $\mathcal{B}_{t}$ leads to
\begin{equation}\label{estimlocT}
e^{\sqrt {c_{0}t}\Lambda_{1}} T_{a}b=\sum_{j\in \mathbb{Z}}W_{j}\,\,\, \mbox{with}\,\,\,
W_{j}\triangleq \mathcal{B}_t(\dot S_{j-1}A,\dot\Delta_jB).
\end{equation}
Notice that $\dot{S}_{j-1}Z_{t,\alpha,\beta}A=\sum\limits_{j'\leq j-2}\dot\Delta_{j'}Z_{t,\alpha,\beta}A$, thanks to properties of operator $K_{\alpha},Z_{t,\alpha, \beta}$,
 we obtain
\begin{eqnarray}\label{cia1}
\|W_{j}\|_{L^q}&\lesssim& \left\{
\begin{array}{l}\sum\limits_{j'\leq j-2}\|\dot\Delta_{j'}Z_{t,\alpha,\beta}A\|_{L^m}\|\dot\Delta_jB\|_{L^p},\,\,\,\,\, q\leq p,\frac{1}{q}=\frac{1}{m}+\frac{1}{p},\\ [1mm]
 \sum\limits_{j'\leq j-2}\|\dot\Delta_{j'}Z_{t,\alpha,\beta}A\|_{L^\infty}\|\dot\Delta_jB\|_{L^q},\,\,\,\,q> p.\\[1mm]
 \end{array} \right.
\end{eqnarray}
If $q\leq p\leq 2q$, by using Bernstein inequality, we have
\begin{eqnarray*}
\|W_{j}\|_{L^q}\lesssim \sum_{j'\leq j-2}2^{(\frac{2d}{p}-\frac{d}{q}-s_{1})j'}2^{s_{1}j'}\|\dot{\Delta}_{j'}A\|_{L^p}\|\dot{\Delta}_{j}B\|_{L^p}.
\end{eqnarray*}
Keep in mind that $s+\frac{2d}{p}-\frac{d}{q}=s_{1}+s_{2}$ and $s_{1}\leq\frac{2d}{p}-\frac{d}{q}$, by Young's inequality for series, we get
\begin{eqnarray*}
\|e^{\sqrt {c_{0}t}\Lambda_1} T_{a}b\|_{\dot{B}^{s}_{q,\infty}}\leq\Big(2^{js}\|W_{j}\|_{L^q}\Big)_{\ell^\infty}
\lesssim\|A\|_{\dot{B}^{s_{1}}_{p,1}} \|B\|_{\dot{B}^{s_{2}}_{p,\infty}}.
\end{eqnarray*}
On the other hand, if $q>p$ then we have
\begin{eqnarray*}
\|W_{j}\|_{L^q}
\lesssim \sum_{j'\leq j-2}2^{(\frac{d}{p}-s_{1})j'}2^{s_{1}j'}\|\dot{\Delta}_{j'}A\|_{L^p}2^{(\frac{d}{p}-\frac{d}{q}-s_{2})j}2^{s_{2}j}\|\dot{\Delta}_{j}B\|_{L^p}.
\end{eqnarray*}
Note that  $s_{1}\leq\frac{d}{p}$, we utilize Young's inequality again to get
\begin{eqnarray*}
\|e^{\sqrt {c_{0}t}\Lambda_1} T_{a}b\|_{\dot{B}^{s}_{q,\infty}}
\lesssim\|A\|_{\dot{B}^{s_{1}}_{p,1}} \|B\|_{\dot{B}^{s_{2}}_{p,\infty}}.
\end{eqnarray*}
As for the paraproduct $T_{b}a$, we proceed the analysis in the similar way and obtain
\begin{eqnarray*}
\|e^{\sqrt {c_{0}t}\Lambda_1} T_{b}a\|_{\dot{B}^{s}_{q,\infty}}
\lesssim\|A\|_{\dot{B}^{s_{1}}_{p,1}} \|B\|_{\dot{B}^{s_{2}}_{p,\infty}},
\end{eqnarray*}
provided that $s_{2}<\min d\big(\frac{1}{p}, \frac{2}{p}-\frac{1}{q}\big)$.

To bound remainder $R(a,b)$, by the spectrum cut-off, one has
$$e^{\sqrt {c_{0}t}\Lambda_1} \dot{\Delta}_{j}R(a,b)=\sum_{j\leq j'+2}Z_{j}\,\,\, \text{with} \,\,\,Z_{j}\triangleq\dot{\Delta}_{j}\mathcal{B}_t(\tilde{\dot{\Delta}}_{j'}a,\dot{\Delta}_{j'}b),$$
where $\tilde{\dot{\Delta}}_{j'}a\triangleq\sum_{|k-j'|\leq1}\dot{\Delta}_{k}a.$ We handle the case $1< p\leq2$ first. By H\"{o}lder and Bernstein inequalities, we arrive at for all $1<p_{0}<q$ that
\begin{eqnarray}\label{siu}
\|Z_{j}\|_{L^q}&=&
\|\dot{\Delta}_{j}\mathcal{B}_t(\tilde{\dot{\Delta}}_{j'}a,\dot{\Delta}_{j'}b)\|_{L^q}\\ \nonumber
&\lesssim&2^{(\frac{d}{p_{0}}-\frac{d}{q})j}\|\dot{\Delta}_{j}\mathcal{B}_t(\tilde{\dot{\Delta}}_{j'}a,\dot{\Delta}_{j'}b)\|_{L^{p_{0}}}\\ \nonumber
&\lesssim&2^{(\frac{d}{p_{0}}-\frac{d}{q})j}\sum_{j'\geq j-2}\|\dot{\Delta}_{j'}A\|_{L^{2p_{0}}}\|\dot{\Delta}_{j'}B\|_{L^{2p_{0}}}\\ \nonumber
&\lesssim&2^{(\frac{d}{p_{0}}-\frac{d}{q})j}\sum_{j'\geq j-2}2^{(\frac{2d}{p}-\frac{d}{p_{0}}-s_{1}-s_{2})j'}2^{s_{1}j'}\|\dot{\Delta}_{j'}A\|_{L^{p}}2^{s_{2}j'}\|\dot{\Delta}_{j'}B\|_{L^p}.
\end{eqnarray}
If $s_{1}+s_{2}>\frac{2d}{p}-d$, then it is able to find some $p_{0}$ close to $1$ such that $s_{1}+s_{2}\geq\frac{2d}{p}-\frac{d}{p_{0}}$,  then one can get
\begin{eqnarray*}
\|e^{\sqrt {c_{0}t}\Lambda_1} R(a,b)\|_{\dot{B}^{s}_{q,\infty}}
\lesssim\|A\|_{\dot{B}^{s_{1}}_{p,1}}\|B\|_{\dot{B}^{s_{2}}_{p,\infty}},
\end{eqnarray*}
where Young's inequality for series was performed.

On the other hand, for $2< p\leq2q$, H\"{o}lder and Bernstein inequalities combined with Lemma \ref{lem5.3} imply that
\begin{eqnarray}\label{siusiu}
\|Z_{j}\|_{L^q}&=&
\|\dot{\Delta}_{j}\mathcal{B}_t(\tilde{\dot{\Delta}}_{j'}a,\dot{\Delta}_{j'}b)\|_{L^q}\\ \nonumber
&\lesssim&2^{(\frac{2d}{p}-\frac{d}{q})j}\|\dot{\Delta}_{j}\mathcal{B}_t(\tilde{\dot{\Delta}}_{j'}a,\dot{\Delta}_{j'}b)\|_{L^\frac{p}{2}}\\ \nonumber
&\lesssim&2^{(\frac{2d}{p}-\frac{d}{q})j}\sum_{j'\geq j-2}\|\dot{\Delta}_{j'}A\|_{L^{p}}\|\dot{\Delta}_{j'}B\|_{L^p}\\ \nonumber
&\lesssim&2^{(\frac{2d}{p}-\frac{d}{q})j}\sum_{j'\geq j-2}2^{(-s_{1}-s_{2})j'}2^{s_{1}j'}\|\dot{\Delta}_{j'}A\|_{L^{p}}2^{s_{2}j'}\|\dot{\Delta}_{j'}B\|_{L^p}.
\end{eqnarray}
Owing to $s_{1}+s_{2}\geq0$, we are led to
\begin{eqnarray*}
\|e^{\sqrt {c_{0}t}\Lambda_1} R(a,b)\|_{\dot{B}^{s}_{q,\infty}}
\lesssim\|A\|_{\dot{B}^{s_{1}}_{p,1}}\|B\|_{\dot{B}^{s_{2}}_{p,\infty}}.
\end{eqnarray*}
Therefore, we achieve (\ref{m1}). The proof of (\ref{m2}) follows the lines from \eqref{estimlocT}-\eqref{siusiu}, which is left to the interesting readers.
\end{proof}

We would like to mention that Lemma \ref{d1} also holds in the framework of Chemin-Lerner's spaces, whereas the time exponent fulfills H\"{o}lder inequality only.
In order to explore the Gevrey decay framework, the following result is crucial, which has been shown recently by \cite{CDX}.
\begin{lem}\label{lem5.1} Let $(r,R)$ satisfy $0<r<R$. For any tempered distribution $u$ fulfilling $\mathrm{supp}\widehat u\subset \lambda C$,
there exists a constant $c>0$ such that for all $\zeta\in \mathbb{R}$ and $\alpha>0$, the following inequality holds for all $p\in [1,\infty]$:
$$\|\Lambda^\zeta e^{-\alpha\Lambda_{1}} u\|_{L^p}\lesssim\lambda^\zeta e^{-c\alpha\lambda}\|u\|_{L^p},$$
where $C(0,r,R)\triangleq \{\xi\in \mathbb{R}^d|r\leq |\xi|\leq R\}$ is the annulus.
\end{lem}

Lemma \ref{lem5.1} enables us to initiate the key idea \eqref{key!} for the time decay, which actually developed the Oliver-Titi's argument to Besov framework. Precisely,
\begin{lem}\label{lem5.2}
Let $\sigma\in \mathbb{R}$ and $1\leq p\leq\infty$. For any $\zeta>0$ and any tempered distribution $u$, it holds for all $t>0$ that
$$\|\Lambda^\zeta u\|^{\ell}_{{\dot B}^{\sigma}_{q,1}}\lesssim t^{-\frac{\zeta}{2}}\|e^{\sqrt {c_{0}t} \Lambda_{1}}u\|^{\ell}_{{\dot B}^{\sigma}_{q,\infty}}$$
and
$$\|\Lambda^\zeta u\|^{h}_{{\dot B}^{\sigma}_{p,1}}\lesssim t^{-\frac{\zeta}{2}}e^{-a\sqrt{t}}\|e^{\sqrt {c_{0}t} \Lambda_{1}}u\|^{h}_{{\dot B}^{\sigma}_{p,\infty}},$$
where $a>0$ is to be confirmed.
\end{lem}
\begin{proof}
By using Lemma \ref{lem5.1} and Bernstein inequality, we have for frequency cut-off $j_{0}$ that
\begin{eqnarray}\label{Gd}
t^\frac{\zeta}{2}\|\Lambda^\zeta u\|^{\ell}_{{\dot B}^{\sigma}_{q,1}}
&=&t^\frac{\zeta}{2}\big\|2^{j\sigma}\|\Lambda^\zeta e^{-\sqrt {c_{0}t} \Lambda_{1}}\Delta_{j}(e^{\sqrt {c_{0}t} \Lambda_{1}}u)\|_{L^q}\big\|_{l^{1}_{j\leq j_{0}}}\\
\nonumber&\lesssim&\big\|(\sqrt t 2^j)^\zeta e^{-c\sqrt {c_{0}t} 2^j} 2^{j\sigma}\|\Delta_{j}(e^{\sqrt {c_{0}t} \Lambda_{1}}u)\|_{L^q}\big\|_{l^{1}_{j\leq j_{0}}}
\nonumber\\ \nonumber&\lesssim&\big(\sum_{j\leq j_{0} }(\sqrt t 2^j)^\zeta e^{-c\sqrt {c_{0}t} 2^j}\big)\|e^{\sqrt {c_{0}t} \Lambda_{1}}u\|^{\ell}_{{\dot B}^{\sigma}_{q,\infty}}\\ \nonumber&\lesssim&\|e^{\sqrt {c_{0}t} \Lambda_{1}}u\|^{\ell}_{{\dot B}^{\sigma}_{q,\infty}},
\end{eqnarray}
where the fact $\sum_{j\in \mathbb{Z} }(\sqrt t 2^j)^\zeta e^{-c\sqrt {c_{0}t} 2^j}\leq C$ for $\zeta>0$ was used in the last inequality.

Furthermore, for the high frequencies, we arrive at
\begin{eqnarray}\label{Gd2}
t^\frac{\zeta}{2}\|\Lambda^\zeta u\|^{h}_{{\dot B}^{\sigma}_{p,1}}
&=&t^\frac{\zeta}{2}\big\|2^{j\sigma}\|\Lambda^\zeta e^{-\sqrt {c_{0}t} \Lambda_{1}}\Delta_{j}(e^{\sqrt {c_{0}t} \Lambda_{1}}u)\|_{L^p}\big\|_{l^{1}_{j\geq j_{0}}}\\
\nonumber&\lesssim&\big\|(\sqrt t 2^j)^\zeta e^{-c\sqrt {c_{0}t} 2^j} 2^{j\sigma}\|\Delta_{j}(e^{\sqrt {c_{0}t} \Lambda_{1}}u)\|_{L^p}\big\|_{l^{1}_{j\geq j_{0}}}
\nonumber\\ \nonumber&\lesssim&\big(\sum_{j\geq j_{0} }(\sqrt t 2^j)^\zeta e^{-\frac{c}{2}\sqrt {c_{0}t} 2^j}\big)e^{-\frac{c}{2}\sqrt {c_{0}t} 2^j}\|e^{\sqrt {c_{0}t} \Lambda_{1}}u\|^{h}_{{\dot B}^{\sigma}_{p,\infty}}\\ \nonumber&\lesssim&e^{-a\sqrt{t}}\|e^{\sqrt {c_{0}t} \Lambda_{1}}u\|^{h}_{{\dot B}^{\sigma}_{p,\infty}},
\end{eqnarray}
where we take $a=\frac{c}{2}\sqrt{c_{0}}2^{j_{0}}$. Hence, the proof of Lemma \ref{lem5.2} is complete.
\end{proof}

Finally, we state the following Gevrey estimate for composition to end this section, which will be frequently used in the sequent analysis.
\begin{prop}\label{prop5.2}
Set $1< p<\infty$, $1\leq  r\leq\infty$. Let $F$ be a real analytic function in a neighborhood of 0, such that
$F(0) = 0$. Let $-\mathrm{min} d(\frac{1}{p},\frac{1}{p'})<s<\frac{d}{p}$ with $\frac{1}{p'}=1-\frac{1}{p}$. There exist two constants $R_0$ and $D$ depending only on $d,p$ and $F$ such that if for some $T>0$,
$$\|e^{\sqrt {c_{0}t}\Lambda_1}z\|_{\tilde{L}^{\infty}_{T}(\dot{B}^{\frac{d}{p}}_{p,1})}\leq R_0$$
then
$$\|e^{\sqrt {c_{0}t}\Lambda_1}F(z)\|_{\tilde{L}^{\theta}_{T}(\dot{B}^{s}_{p,r})}\leq D\|e^{\sqrt {c_{0}t}\Lambda_1}z\|_{\tilde{L}^{\theta}_{T}(\dot{B}^{s}_{p,r})}$$
for $1\leq \theta\leq\infty$. Moreover, the case $s=\frac{d}{p}$ holds true for $r=1$.
\end{prop}

\section{Uniform bounds on the growth of the radius of analyticity}
In this section, we establish uniform bounds on the growth of the radius of analyticity of the
solution in negative Besov norms, which is the main part in the proof of Theorem \ref{thm2.3}. More precisely, it is shown that there exists a positive constant $c_{0}$ depending only on $d, \bar{\mu}, \bar{\lambda}, \bar{\kappa},\rho^{*}$ such that for any $T>0$
\begin{eqnarray}\label{evoo}
\sup_{t\in [0,T]}\|e^{\sqrt {c_{0}t} \Lambda_{1}}(\nabla a, m)\|^{\ell}_{{\dot B}^{-\sigma_{1}-1}_{q,\infty}}\leq C_0
\end{eqnarray}
where $C_0>0$ depends on the initial norm $\|(\nabla a_{0},m_{0})\|^{\ell}_{\dot{B}^{-\sigma_{1}-1}_{q,\infty}}$. Indeed, (\ref{evoo}) is fulfilled by Theorem \ref{thm2.1} if $\sigma_{1}=2-\frac{d}{q}$, thus in what follows, we prove (\ref{evoo}) with $\sigma_{1}>2-\frac{d}{q}$.

For clarity, we agree with  $(A,U,G)\triangleq (e^{\sqrt {c_{0}t}\Lambda_1}{a}, e^{\sqrt {c_{0}t}\Lambda_1}{m}, e^{\sqrt {c_{0}t}\Lambda_1}{g})$ and denote
the energy functional by
\begin{eqnarray}\label{SX}
\bar{X}_{p,q}(T)\triangleq
\|( A,U)\|_{E^{p,q}_{T}}.
 \end{eqnarray}

We establish the following lemma which is dedicated to a priori estimate on the evolution of Gevrey regularity.
\begin{lem}\label{tol}
Let $\sigma_{1}$ fulfills $\sigma_{1}<d-\frac{d}{q}$ if $1<p\leq2$ while $\sigma_{1}\leq\frac{2d}{p}-\frac{d}{q}$ if $p>2$.
It holds that
\begin{multline}\label{energy esti}
\|(\nabla A,U)\|^{\ell}_{\tilde{L}^\infty_{T}(\dot{B}^{-\sigma_{1}-1}_{q,\infty})\cap \tilde{L}^1_{T}(\dot{B}^{-\sigma_{1}+1}_{q,\infty})} \lesssim\|(\nabla a_{0},m_{0})\|^{\ell}_{\dot{B}^{-\sigma_{1}-1}_{q,\infty}}+\big(1+\bar{X}_{p,q}(T)\big)\bar{X}^{2}_{p,q}(T)\\+\big(1+\bar{X}_{p,q}(T)\big)\bar{X}_{p,q}(T)\|(\nabla A,U)\|^{\ell}_{\tilde{L}^\infty_{T}(\dot{B}^{-\sigma_{1}-1}_{q,\infty})\cap \tilde{L}^1_{T}(\dot{B}^{-\sigma_{1}+1}_{q,\infty})}.
\end{multline}
\end{lem}

\begin{proof}
\underline{\textit{Step 1: Linear analysis}}

In light of classical Leray Projector $\mathcal{P}\triangleq \mathrm{Id}-\frac{\nabla}{\Delta}\mathrm{div}$, one can write $m=\mathcal{P}m+\cQ m$, which is the sum of the incompressible part $\mathcal{P}m$ and the compressible part $\cQ m$. Note that $\nabla\div m=\Delta\cQ m$, it follows from (\ref{linearized}) that compressible part $\cQ m$ and $a$ that
$$\left\{\begin{array}{l}\d_t\nabla a+ \Delta\cQ m=0,\\[1ex]
\d_t\cQ m-\bar{\nu}\Delta\cQ m-\bar{\kappa}\Delta\nabla  a=\cQ g,
\end{array}\right.$$
where $\bar{\nu}\triangleq2\bar{\mu}+\bar{\lambda}$. In order to handle the coupling between $a$ and $\cQ m,$ inspired by \cite{H2}, let us introduce the effective velocity
$$w\triangleq \cQ m+\alpha \nabla a$$
where $\alpha=\frac{1}{2}(\bar{\nu}\pm\sqrt{\bar{\nu}^2-4\bar{\kappa}})$. Consequently, we arrive at
\begin{equation}
\d_tw-(\bar{\nu}-\alpha)\Delta w=\cQ g\label{4.2}
\end{equation}
with $\bar{\nu}-\alpha=\frac{1}{2}(\bar{\nu}\mp\sqrt{\bar{\nu}^2-4\bar{\kappa}})>0$. Denote $W\triangleq e^{\sqrt {c_{0}t}\Lambda_1}{w}$. Applying Lemma \ref{lem4.2} (taking $\sigma=-\sigma_{1}-1$ and $r=\infty$) indicates that
\begin{equation}\label{Cal}
\|W\|^{\ell}_{\tilde{L}_{T}^{\infty}(\dot{B}^{-\sigma_{1}-1}_{q,\infty})}+
\|W\|^{\ell}_{\tilde{L}_{T}^{1}(\dot{B}^{-\sigma_{1}+1}_{q,\infty})}\lesssim  \|w_0\|^{\ell}_{\dot{B}^{-\sigma_{1}-1}_{q,\infty}}+\|\cQ
G\|^{\ell}_{\tilde{L}_{T}^{1}(\dot{B}^{-\sigma_{1}-1}_{q,\infty})}.
\end{equation}
On the other hand, bounding $\mathcal{Q}m$, we use the definition of $w$ so that the equation for $\mathcal{Q}m$ can be written as
\begin{equation}\label{Qm}
\d_t\cQ m-\alpha \Delta \cQ m=\frac{\bar{\kappa}}{\alpha}\Delta w+\cQ g.
\end{equation}
Hence, one use Lemma \ref{lem4.2} again and arrive at
\begin{equation*}
\|\mathcal{Q}U\|^{\ell}_{\tilde{L}_{T}^{\infty}(\dot{B}^{-\sigma_{1}-1}_{q,\infty})}+
\|\mathcal{Q}U\|^{\ell}_{\tilde{L}_{T}^{1}(\dot{B}^{-\sigma_{1}+1}_{q,\infty})}\lesssim
\|\mathcal{Q}m_0\|^{\ell}_{\dot{B}^{-\sigma_{1}-1}_{q,\infty}}+\|w_{0}\|^{\ell}_{\dot{B}^{-\sigma_{1}-1}_{q,\infty}}
+\|\cQ G\|^{\ell}_{\tilde{L}_{T}^{1}(\dot{B}^{-\sigma_{1}-1}_{q,\infty})}.
\end{equation*}
Consequently, it follows from the fact $\nabla a=\frac{w-\cQ m}{\alpha}$ that
\begin{multline}\label{Gh2}
\|(\nabla A,\mathcal{Q}U)\|^{\ell}_{\tilde{L}_{T}^{\infty}(\dot{B}^{-\sigma_{1}-1}_{q,\infty})}+\|(\nabla A,\mathcal{Q}U)\|^{\ell}_{\tilde{L}_{T}^{1}(\dot{B}^{-\sigma_{1}+1}_{q,\infty})}\\ \lesssim
\|(\nabla a_{0},\mathcal{Q}m_{0})\|^{\ell}_{\dot{B}^{-\sigma_{1}-1}_{q,\infty}}+
\|\mathcal{Q}G\|^{\ell}_{\tilde{L}_{T}^{1}(\dot{B}^{-\sigma_{1}-1}_{q,\infty})}.
\end{multline}

In addition, regarding the incompressible part $\mathcal{P}m$, we see that
\begin{equation}\label{incompressible}\partial_{t}\mathcal{P}m-\bar{\mu}\Delta \mathcal{P}m =\mathcal{P}g.\end{equation}
Hence, applying Lemma \ref{lem4.2} to (\ref{incompressible}) yields
\begin{equation}\label{4.8r}
\|\mathcal{P}U\|^{\ell}_{\tilde{L}_{T}^{\infty}(\dot{B}^{-\sigma_{1}-1}_{q,\infty})}+\|\mathcal{P}U\|^{\ell}_{\tilde{L}_{T}^{1}(\dot{B}^{-\sigma_{1}+1}_{q,\infty})}\lesssim
\|\mathcal{P}m_{0}\|^{\ell}_{\dot{B}^{-\sigma_{1}-1}_{q,\infty}}+
\|\mathcal{P}G\|^{\ell}_{\tilde{L}_{T}^{1}(\dot{B}^{-\sigma_{1}-1}_{q,\infty})}.
\end{equation}
Combining (\ref{Gh2}) with (\ref{4.8r}) leads to
\begin{multline}\label{Gh}
\|(\nabla A,U)\|^{\ell}_{\tilde{L}_{T}^{\infty}(\dot{B}^{-\sigma_{1}-1}_{q,\infty})}+\|(\nabla A,U)\|^{\ell}_{\tilde{L}_{T}^{1}(\dot{B}^{-\sigma_{1}+1}_{q,\infty})}\lesssim
\|(\nabla a_{0},m_{0})\|^{\ell}_{\dot{B}^{-\sigma_{1}-1}_{q,\infty}}+
\|G\|^{\ell}_{\tilde{L}_{T}^{1}(\dot{B}^{-\sigma_{1}-1}_{q,\infty})}.
\end{multline}

\underline{\textit{Step 2: Nonlinear estimates}}

Next, we turn to bound nonlinear term $\|G\|^{\ell}_{\tilde{L}_{T}^{1}(\dot{B}^{-\sigma_{1}-1}_{q,\infty})}$. Here, we would like to present the explicit formulation of $g(a,m)$ (see \cite{SX}). Precisely, one can write source terms $g(a,m)=\sum\limits^{6}_{i=1} g_{i}(a,m)$ as follows:
\begin{equation}
\left\{
\begin{array}{l}g_{1}=\mathrm{div}\big((Q(a)-1)m\otimes m\big),\\ [1mm]
g_{2}=\bar{\mu}\Delta\big(Q(a)m\big)+(\bar{\mu}+\bar{\lambda})\nabla\mathrm{div}\big(Q(a)m\big),\\ [1mm]
g_{3}=2\mathrm{div}\Big(\tilde{\mu}(a)D\big((1-Q(a))m\big)\Big)+\nabla\Big(\tilde{\lambda}(a)\mathrm{div}\big((1-Q(a))m\big)\Big),\\ [1mm]
g_{4}=\nabla\big(aL(a)\big),\\ [1mm]
g_{5}=\nabla\big(\tilde{\kappa}_{1}(a)\Delta a\big),\\ [1mm]
g_{6}=\frac{\rho^{*}}{2}\nabla\Big(\big(\tilde{\kappa}_{2}(a)+\check{\kappa}\big)|\nabla a|^{2}\Big)-\mathrm{div}\Big(\big(\tilde{\kappa}_{3}(a)+\bar{\kappa}\big)\nabla
a\otimes\nabla a\Big),\\ [1mm]
 \end{array} \right.\label{nonlinear}
\end{equation}
with $\check{\kappa}\triangleq\kappa(\rho^{*})+\rho^{*}\kappa'(\rho^{*})$ and the corresponding composite functions are given by
\begin{equation}
\left\{
\begin{array}{l}Q(a)=\frac{a}{a+1},\\ [1mm]
L(a)=P'(\rho^{*}+\theta \rho^{*}a),\\ [1mm]
\tilde{\mu}(a)=\frac{2\mu(a\rho^{*}+\rho^{*})}{\rho^{*}}-2\bar{\mu},\\ [1mm]
\tilde{\lambda}(a)=\frac{\lambda(a\rho^{*}+\rho^{*})}{\rho^{*}}-\bar{\lambda},\\ [1mm]
\tilde{\kappa}_{1}(a)=(a\rho^{*}+\rho^{*})\kappa(a\rho^{*}+\rho^{*})-\bar{\kappa},\\ [1mm]
\tilde{\kappa}_{2}(a)=\kappa(a\rho^{*}+\rho^{*})+(a\rho^{*}+\rho^{*})\kappa'(a\rho^{*}+\rho^{*})-\check{\kappa},\\ [1mm]
\tilde{\kappa}_{3}(a)=\rho^{*}\kappa(a\rho^{*}+\rho^{*})-\bar{\kappa}.\\ [1mm]
 \end{array} \right.\label{composite1}
\end{equation}
Here, those functions in (\ref{composite1}) are assumed to be analytic and to be vanish at zero.

To bound nonlinear terms $g_1\sim g_6$, we first claim the following two
non classical product estimates:
\begin{eqnarray}\label{mark2}
\|e^{\sqrt {c_{0}t}\Lambda_{1}}(ab)\|_{\tilde{L}^{\rho}_{T}(\dot{B}^{-\sigma_{1}}_{q,\infty})}\lesssim\|A\|_{\tilde{L}^{\rho_{1}}_{T}(\dot{B}^{\frac{d}{p}-1}_{p,1})}
\|B\|_{\tilde{L}^{\rho_{2}}_{T}(\dot{B}^{1-\sigma_{1}+\frac{d}{p}-\frac{d}{q}}_{p,\infty})},
\end{eqnarray}
\begin{eqnarray}\label{mark1}
\|e^{\sqrt {c_{0}t}\Lambda_{1}}(ab)\|_{\tilde{L}^{\rho}_{T}(\dot{B}^{-\sigma_{1}}_{q,\infty})}\lesssim\|A\|_{\tilde{L}^{\rho_{1}}_{T}(\dot{B}^{\frac{d}{p}}_{p,1})}
\|B\|_{\tilde{L}^{\rho_{2}}_{T}(\dot{B}^{-\sigma_{1}}_{q,\infty})}
\end{eqnarray}
for all  $\rho,\rho_{1},\rho_{2}\in[1, \infty]$ satisfying $\frac{1}{\rho}=\frac{1}{\rho_{1}}+\frac{1}{\rho_{2}}$.

Indeed, recalling (\ref{exist}) and $2-\frac{d}{q}<\sigma_{1}$, one has
$$1-\sigma_{1}+\frac{d}{p}-\frac{d}{q}<\frac{d}{p}-1\leq\frac{2d}{p}-\frac{d}{q}.$$
Also, due to the assumption that $\sigma_{1}< d-\frac{d}{q}$ for $1< p\leq2$ and $\sigma_{1}\leq \frac{2d}{p}-\frac{d}{q}$ for $p>2$, it holds that
\begin{equation}
\left\{
\begin{array}{l}-\sigma_{1}+ \frac{2d}{p}-\frac{d}{q}\geq0,\,\,\,\,\,\,\,\,\,\,\,\,\,\,\,\,\,p>2;\\ [1mm]
-\sigma_{1}+ \frac{2d}{p}-\frac{d}{q}>\frac{2d}{p}-d,\,\,\,\,1< p\leq2.\\ [1mm]
 \end{array} \right.\label{composite2}
\end{equation}
Consequently, (\ref{mark2}) stems from (\ref{m1}) in Lemma  \ref{d1} with $s=-\sigma_{1}$, $s_{1}=\frac{d}{p}-1, s_{2}=1-\sigma_{1}+\frac{d}{p}-\frac{d}{q}$ and $q\leq p$. As for (\ref{mark1}), it is actually followed by (\ref{m2}) in Lemma  \ref{d1} with
$s=-\sigma_{1}, s_{1}=\frac{d}{p}, s_{2}=-\sigma_{1}$
satisfying $-\sigma_{1}<\frac{d}{q}-1\leq\frac{d}{p}$ and $q\leq p$.

Hence, we start with bounding the first term $G_{1}\triangleq e^{\sqrt {c_{0}t}\Lambda_{1}}g_1$. It is
convenient to decompose
$$\mathrm{div}(m\otimes m)=\mathrm{div}(m^{h}\otimes m)+\mathrm{div}(m^{\ell}\otimes m).$$
For the high frequency part, it is shown by (\ref{mark2}) that
\begin{multline}\label{hosto}
\|e^{\sqrt {c_{0}t}\Lambda_{1}}\mathrm{div}(m^{h}\otimes m)\|^{\ell}_{\tilde{L}^{1}_{T}(\dot{B}^{-\sigma_{1}-1}_{q,\infty})}\lesssim
\|U\|_{\tilde{L}^2_{T}(\dot{B}^{\frac{d}{p}-1}_{p,1})}\|U^{h}\|_{\tilde{L}^2_{T}(\dot{B}^{1-\sigma_{1}+\frac{d}{p}-\frac{d}{q}}_{p,\infty})}\\
\lesssim
(\|U^{h}\|_{\tilde{L}^2_{T}(\dot{B}^{\frac{d}{p}}_{p,1})}+\|U^{\ell}\|_{\tilde{L}^2_{T}(\dot{B}^{\frac{d}{q}-2}_{q,1})})
\|U^{h}\|_{\tilde{L}^2_{T}(\dot{B}^{\frac{d}{p}}_{p,1})},
\end{multline}
where the fact that $1-\sigma_{1}+\frac{d}{p}-\frac{d}{q}<\frac{d}{p}$ was used in the last inequality. Thanks to the interpolation (see \cite{BCD})
$$\tilde{L}^2_{T}(\dot{B}^{\frac{d}{p}}_{p,1})=\Big(\tilde{L}^{\infty}_{T}(\dot{B}^{\frac{d}{p}-1}_{p,1}),\tilde{L}^1_{T}(\dot{B}^{\frac{d}{p}+1}_{p,1})\Big)_{\frac{1}{2}}; \quad
\tilde{L}^2_{T}(\dot{B}^{\frac{d}{q}-2}_{q,1})=\Big(\tilde{L}^{\infty}_{T}(\dot{B}^{\frac{d}{q}-3}_{q,\infty}),\tilde{L}^1_{T}(\dot{B}^{\frac{d}{q}-1}_{q,\infty})\Big)_{\frac{1}{2}}
$$
and Theorem \ref{thm2.1}, furthermore, we arrive at
\begin{eqnarray}\label{R-E3.28}
\|e^{\sqrt {c_{0}t}\Lambda_{1}}\mathrm{div}(m^{h}\otimes m)\|^{\ell}_{\tilde{L}^{1}_{T}(\dot{B}^{-\sigma_{1}-1}_{q,\infty})}\lesssim
\bar{X}^{2}_{p,q}(T).
\end{eqnarray}
For the low frequency term with $m^{\ell}$, it follows from (\ref{mark1}) and the Sobolev embedding that
\begin{multline}\label{R-E3.29}
\|e^{\sqrt {c_{0}t}\Lambda_{1}}\mathrm{div}(m^{\ell}\otimes m)\|^{\ell}_{\tilde{L}^{1}_{T}(\dot{B}^{-\sigma_{1}-1}_{q,\infty})}\lesssim\|U\|_{\tilde{L}^1_{T}(\dot{B}^{\frac{d}{p}}_{p,1})}
\|U^{\ell}\|_{\tilde{L}^\infty_{T}(\dot{B}^{-\sigma_{1}}_{q,\infty})}\\
\lesssim(\|U\|^{h}_{\tilde{L}^{1}_{T}(\dot{B}^{\frac{d}{p}+1}_{p,1})}+\|U\|^{\ell}_{\tilde{L}^{1}_{T}(\dot{B}^{\frac{d}{q}-1}_{q,\infty})})
\|U\|^{\ell}_{\tilde{L}^{\infty}_{T}(\dot{B}^{-\sigma_{1}-1}_{q,\infty})}
\lesssim\bar{X}_{p,q}(T)\|U\|^{\ell}_{\tilde{L}^{\infty}_{T}(\dot{B}^{-\sigma_{1}-1}_{q,\infty})}.
\end{multline}
So we obtain
\begin{eqnarray}\label{R-E3.30}
\|e^{\sqrt {c_{0}t}\Lambda_{1}}\mathrm{div}(m\otimes m)\|^{\ell}_{\tilde{L}^{1}_{T}(\dot{B}^{-\sigma_{1}-1}_{q,\infty})}\lesssim\bar{X}^{2}_{p,q}(T)+\bar{X}_{p,q}(T)\|U\|^{\ell}_{\tilde{L}^{\infty}_{T}(\dot{B}^{-\sigma_{1}-1}_{q,\infty})}.
\end{eqnarray}
In addition, the composite term in $G_{1}$ can be similarly treated. Actually, (\ref{mark1}) and Proposition \ref{prop5.2} (as $\|A\|_{\tilde{L}^\infty_{T}(\dot{B}^{\frac{d}{p}}_{p,1})}$ is bounded by Theorem \ref{thm2.1}) imply that
\begin{multline}\label{R-E3.301}
\|e^{\sqrt {c_{0}t}\Lambda_{1}}\mathrm{div}(Q(a)m\otimes m)\|^{\ell}_{\tilde{L}^{1}_{T}(\dot{B}^{-\sigma_{1}-1}_{q,\infty})}\\
\lesssim\|e^{\sqrt {c_{0}t}\Lambda_{1}}Q(a)\|_{\tilde{L}^\infty_{T}(\dot{B}^{\frac{d}{p}}_{p,1})}\|e^{\sqrt {c_{0}t}\Lambda_{1}}(m\otimes m)\|_{\tilde{L}^{1}_{T}(\dot{B}^{-\sigma_{1}}_{q,\infty})}\\
\lesssim\|A\|_{\tilde{L}^\infty_{T}(\dot{B}^{\frac{d}{p}}_{p,1})}\|e^{\sqrt {c_{0}t}\Lambda_{1}}(m\otimes m)\|_{\tilde{L}^{1}_{T}(\dot{B}^{-\sigma_{1}}_{q,\infty})}.
\end{multline}
 Hence, employing those calculations as (\ref{hosto})-(\ref{R-E3.29}) to (\ref{R-E3.301}), one has
 \begin{eqnarray} \label{R-E3.3111}
\hspace{10mm} \|e^{\sqrt {c_{0}t}\Lambda_{1}}\mathrm{div}(Q(a)m\otimes m)\|^{\ell}_{\tilde{L}^{1}_{T}(\dot{B}^{-\sigma_{1}-1}_{q,\infty})}
\lesssim\bar{X}^{3}_{p,q}(T)+\bar{X}^{2}_{p,q}(T)\|U\|^{\ell}_{\tilde{L}^{\infty}_{T}(\dot{B}^{-\sigma_{1}-1}_{q,\infty})}.
\end{eqnarray}
Consequently, we add (\ref{R-E3.3111}) to (\ref{R-E3.30}) and conclude that
\begin{multline}\label{R-E3.31}
\|G_{1}(t)\|^{\ell}_{\tilde{L}^{1}_{T}(\dot{B}^{-\sigma_{1}-1}_{q,\infty})}\lesssim\big(1+\bar{X}_{p,q}(T)\big)\bar{X}^{2}_{p,q}(T)
+\big(1+\bar{X}_{p,q}(T)\big)\bar{X}_{p,q}(T)\|U\|^{\ell}_{\tilde{L}^{\infty}_{T}(\dot{B}^{-\sigma_{1}-1}_{q,\infty})}.
\end{multline}
For $G_{2}\triangleq e^{\sqrt {c_{0}t}\Lambda_{1}} g_2$, it suffices to consider the first term $e^{\sqrt {c_{0}t}\Lambda_{1}}\Delta(Q(a)m)$, since another term can be proceeded in a similar way. Precisely,
\begin{multline}\label{R-E3.32}
\|e^{\sqrt {c_{0}t}\Lambda_{1}}\Delta\big(Q(a)m\big)\|^{\ell}_{\tilde{L}^{1}_{T}(\dot{B}^{-\sigma_{1}-1}_{q,\infty})}
\lesssim\|e^{\sqrt {c_{0}t}\Lambda_{1}}\big(Q(a)\nabla m\big)\|^{\ell}_{\tilde{L}^{1}_{T}(\dot{B}^{-\sigma_{1}}_{q,\infty})}\\+\|e^{\sqrt {c_{0}t}\Lambda_{1}}\big(\nabla Q(a) m\big)\|^{\ell}_{\tilde{L}^{1}_{T}(\dot{B}^{-\sigma_{1}}_{q,\infty})}
\triangleq K_{1}+K_{2}.
\end{multline}
To handle $K_{1}$, we write $Q(a)\nabla m=Q(a)\nabla m^{h}+Q(a)\nabla m^{\ell}$. It follows from
Proposition \ref{prop5.2} and (\ref{mark2}) that
\begin{multline}\label{monet}
\|e^{\sqrt {c_{0}t}\Lambda_{1}}\big(Q(a)\nabla m^{h}\big)\|^{\ell}_{\tilde{L}^1_{T}(\dot{B}^{-\sigma_{1}}_{q,\infty})}\lesssim\|A\|_{\tilde{L}^\infty_{T}(\dot{B}^{\frac{d}{p}-1}_{p,1})}\|\nabla U^{h}\|_{\tilde{L}^1_{T}(\dot{B}^{1-\sigma_{1}+\frac{d}{p}-\frac{d}{q}}_{p,\infty})}\\
\lesssim(\|A\|^{h}_{\tilde{L}^\infty_{T}(\dot{B}^{\frac{d}{p}}_{p,1})}+\|A\|^{\ell}_{\tilde{L}^\infty_{T}(\dot{B}^{\frac{d}{q}-2}_{q,\infty})})
\|U\|^{h}_{\tilde{L}^1_{T}(\dot{B}^{\frac{d}{p}+1}_{p,1})}.
\end{multline}
It is also easy to show that by (\ref{mark1})
\begin{multline}\label{boboka}
\|e^{\sqrt {c_{0}t}\Lambda_{1}}\big(Q(a)\nabla m^{\ell}\big)\|^{\ell}_{\tilde{L}^1_{T}(\dot{B}^{-\sigma_{1}}_{q,\infty})}\lesssim\|A\|_{\tilde{L}^\infty_{T}(\dot{B}^{\frac{d}{p}}_{p,1})}\|\nabla U^{\ell}\|_{\tilde{L}^1_{T}(\dot{B}^{-\sigma_{1}}_{q,\infty})}\\
\lesssim\big(\| A\|^{h}_{\tilde{L}^\infty_{T}(\dot{B}^{\frac{d}{p}}_{p,1})}+\| A\|^{\ell}_{\tilde{L}^\infty_{T}(\dot{B}^{\frac{d}{q}-2}_{q,\infty})}\big)\|U\|^{\ell}_{\tilde{L}^1_{T}(\dot{B}^{-\sigma_{1}+1}_{q,\infty})}.
\end{multline}
As for $K_{2}$, we have
\begin{multline}\label{piqiu}
\|e^{\sqrt {c_{0}t}\Lambda_{1}}\big(\nabla Q(a)m\big)\|^{\ell}_{\tilde{L}^1_{T}(\dot{B}^{-\sigma_{1}}_{q,\infty})}\lesssim\| A\|_{\tilde{L}^2_{T}(\dot{B}^{\frac{d}{p}}_{p,1})}\|U\|_{\tilde{L}^2_{T}(\dot{B}^{1-\sigma_{1}+\frac{d}{p}-\frac{d}{q}}_{p,\infty})}\\
\lesssim(\|A^{h}\|_{\tilde{L}^2_{T}(\dot{B}^{\frac{d}{p}+1}_{p,1})}+\|A^{\ell}\|_{\tilde{L}^2_{T}(\dot{B}^{\frac{d}{q}-1}_{q,1})})
(\|U^{h}\|_{\tilde{L}^2_{T}(\dot{B}^{\frac{d}{p}}_{p,1})}+\|U^{\ell}\|_{\tilde{L}^2_{T}(\dot{B}^{-\sigma_{1}}_{q,\infty})}).
\end{multline}
Owing to the interpolation
\begin{eqnarray}\label{tuo}\|(\nabla A, U)^{\ell}\|_{\tilde{L}^2_{T}(\dot{B}^{-\sigma_{1}}_{q,\infty})}\lesssim\Big(\| (\nabla A, U)^{\ell}\|_{\tilde{L}^\infty_{T}(\dot{B}^{-\sigma_{1}-1}_{q,\infty})}\Big)^\frac{1}{2}
\Big(\| (\nabla A, U)^{\ell}\|_{\tilde{L}^1_{T}(\dot{B}^{-\sigma_{1}+1}_{q,\infty})}\Big)^\frac{1}{2},\end{eqnarray}
hence, (\ref{monet})-(\ref{tuo}) enables us to get
\begin{eqnarray} \label{R-E3.37}
\hspace{10mm}\|G_{2}(t)\|^{\ell}_{\tilde{L}^1_{T}(\dot{B}^{-\sigma_{1}-1}_{q,\infty})}\lesssim\bar{X}^{2}_{p,q}(T)+\bar{X}_{p,q}(T)\|(\nabla A,U)\|^{\ell}_{\tilde{L}^\infty_{T}(\dot{B}^{-\sigma_{1}-1}_{q,\infty})\cap \tilde{L}^1_{T}(\dot{B}^{-\sigma_{1}+1}_{q,\infty})}.
\end{eqnarray}

Bounding $G_{3}\triangleq e^{\sqrt {c_{0}t}\Lambda_{1}} g_3$ is close to $G_{2}$. It is enough to bound $$e^{\sqrt {c_{0}t}\Lambda_{1}}\mathrm{div}\big(\tilde{\mu}(a)D\big(Q(a)m\big)\big)$$ and other terms follow from similar steps. Keeping in mind that (\ref{mark1}), one has
\begin{multline}
\big\|e^{\sqrt {c_{0}t}\Lambda_{1}}\mathrm{div}\big(\tilde{\mu}(a)D\big(Q(a)m\big)\big)\big\|^{\ell}_{\tilde{L}^1_{T}(\dot{B}^{-\sigma_{1}-1}_{q,\infty})}
\lesssim\|A\|_{\tilde{L}^\infty_{T}(\dot{B}^{\frac{d}{p}}_{p,1})}
\|e^{\sqrt {c_{0}t}\Lambda_{1}}D\big(Q(a)m\big)\|_{\tilde{L}^1_{T}(\dot{B}^{-\sigma_{1}}_{q,\infty})}.
\end{multline}
Thus employing the similar estimates as (\ref{R-E3.32})-(\ref{R-E3.37}) to get
\begin{multline}
\big\|e^{\sqrt {c_{0}t}\Lambda_{1}}\mathrm{div}\big(\tilde{\mu}(a)D\big(Q(a)m\big)\big)\big\|^{\ell}_{\tilde{L}^1_{T}(\dot{B}^{-\sigma_{1}-1}_{q,\infty})}\\
\lesssim\bar{X}^{3}_{p,q}(T)+\bar{X}^{2}_{p,q}(T)\|(\nabla A,U)\|^{\ell}_{\tilde{L}^\infty_{T}(\dot{B}^{-\sigma_{1}-1}_{q,\infty})\cap \tilde{L}^1_{T}(\dot{B}^{-\sigma_{1}+1}_{q,\infty})}
\end{multline}
which indicates that
\begin{multline}\label{R-E3.388}
\|G_{3}(t)\|^{\ell}_{\tilde{L}^1_{T}(\dot{B}^{-\sigma_{1}-1}_{q,\infty})}\lesssim\big(1+\bar{X}_{p,q}(T)\big)\bar{X}^{2}_{p,q}(T)\\
+\big(1+\bar{X}_{p,q}(T)\big)\bar{X}_{p,q}(T)\|(\nabla A,U)\|^{\ell}_{\tilde{L}^{\infty}_{T}(\dot{B}^{-\sigma_{1}-1}_{q,\infty})\cap \tilde{L}^1_{T}(\dot{B}^{-\sigma_{1}+1}_{q,\infty})}.
\end{multline}

Next, we turn to estimate the pressure term $G_{4}\triangleq e^{\sqrt {c_{0}t}\Lambda_{1}} g_4$. We use the decomposition $\nabla(aL(a))=\nabla(a^{h}L(a))+\nabla(a^{\ell}L(a))$. Observe that $1-\sigma_{1}+\frac{d}{p}-\frac{d}{q}>-d\mathrm{min}(\frac{1}{p},\frac{1}{p'})$ with $\frac{1}{p}+\frac{1}{p'}=1$, it follows from Proposition \ref{prop5.2} and (\ref{mark2}) that
\begin{multline}\label{R-E3.38}
\|e^{\sqrt {c_{0}t}\Lambda_{1}}\nabla(a^{h}L(a))\|^{\ell}_{\tilde{L}^1_{T}(\dot{B}^{-\sigma_{1}-1}_{q,\infty})}\lesssim\|A\|_{\tilde{L}^\infty_{T}(\dot{B}^{1-\sigma_{1}+\frac{d}{p}-\frac{d}{q}}_{p,\infty})}
\|A\|^{h}_{\tilde{L}^1_{T}(\dot{B}^{\frac{d}{p}-1}_{p,1})}\\
\lesssim
\big(\|A\|^{h}_{\tilde{L}^\infty_{T}(\dot{B}^{\frac{d}{p}}_{p,1})}+\|A\|^{\ell}_{\tilde{L}^\infty_{T}(\dot{B}^{-\sigma_{1}}_{q,\infty})}\big)
\|A\|^{h}_{\tilde{L}^1_{T}(\dot{B}^{\frac{d}{p}+2}_{p,1})}.
\end{multline}
The corresponding low-frequency term can be estimated as
\begin{multline}\label{R-E3.39}
\|e^{\sqrt {c_{0}t}\Lambda_{1}}\nabla(a^{\ell}L(a))\|^{\ell}_{\tilde{L}^1_{T}(\dot{B}^{-\sigma_{1}-1}_{q,\infty})}
\lesssim\|A^{\ell}\|_{\tilde{L}^2_{T}(\dot{B}^{1-\sigma_{1}+\frac{d}{p}-\frac{d}{q}}_{p,\infty})}\|A\|_{\tilde{L}^{2}_{T}(\dot{B}^{\frac{d}{p}-1}_{p,1})}\\
\lesssim\Big(\|A^{h}\|_{\tilde{L}^{2}_{T}(\dot{B}^{\frac{d}{p}+1}_{p,1})}+\|A^{\ell}\|_{\tilde{L}^{2}_{T}(\dot{B}^{\frac{d}{q}-1}_{q,1})}\Big)
\|A^{\ell}\|_{\tilde{L}^2_{T}(\dot{B}^{-\sigma_{1}+1}_{q,1})}.
\end{multline}
Due to the interpolation (\ref{tuo}), furthermore,  we have
\begin{eqnarray}\label{R-E3.40}
&&\hspace{5mm}\|e^{\sqrt {c_{0}t}\Lambda_{1}}\nabla(a^{\ell}L(a))\|^{\ell}_{\tilde{L}^1_{T}(\dot{B}^{-\sigma_{1}-1}_{q,\infty})}
\lesssim\bar{X}_{p,q}(T)\big(\|\nabla A\|^{\ell}_{\tilde{L}^\infty_{T}(\dot{B}^{-\sigma_{1}-1}_{q,\infty})}
+\|\nabla A\|^{\ell}_{\tilde{L}^1_{T}(\dot{B}^{-\sigma_{1}+1}_{q,\infty})}\big).
\end{eqnarray}
Thus we conclude that
\begin{eqnarray}\label{pressure}
&&\|G_{4}(t)\|^{\ell}_{\tilde{L}^{1}_{T}(\dot{B}^{-\sigma_{1}-1}_{q,\infty})}\lesssim\bar{X}^{2}_{p,q}(T)+\bar{X}_{p,q}(T)\|\nabla A\|^{\ell}_{\tilde{L}^\infty_{T}(\dot{B}^{-\sigma_{1}-1}_{q,\infty})\cap\tilde{L}^1_{T}(\dot{B}^{-\sigma_{1}+1}_{q,\infty})}.
\end{eqnarray}

Let us look at Korteweg terms and bound $G_{5}\triangleq e^{\sqrt {c_{0}t}\Lambda_{1}} g_5$ first. Thanks to  (\ref{mark2}), we deduce
\begin{multline*}
\|G_{5}(t)\|^{\ell}_{\tilde{L}^1_{T}(\dot{B}^{-\sigma_{1}-1}_{q,\infty})}\lesssim\|A\|_{\tilde{L}^\infty_{T}(\dot{B}^{1-\sigma_{1}+\frac{d}{p}-\frac{d}{q}}_{p,\infty})}
\|\Delta A\|_{\tilde{L}^1_{T}(\dot{B}^{\frac{d}{p}-1}_{p,1})}\\
\lesssim\big(\|A\|^{h}_{\tilde{L}^\infty_{T}(\dot{B}^{\frac{d}{p}}_{p,1})}+\|A\|^{\ell}_{\tilde{L}^\infty_{T}(\dot{B}^{-\sigma_{1}}_{q,\infty})}\big)
(\|A\|^{h}_{\tilde{L}^1_{T}(\dot{B}^{\frac{d}{p}+2}_{p,1})}+
\|A\|^{\ell}_{\tilde{L}^1_{T}(\dot{B}^{\frac{d}{q}}_{q,\infty})}),
\end{multline*}
which implies that
\begin{eqnarray}\label{R-EKort}
\|G_{5}(t)\|^{\ell}_{\tilde{L}_{T}^{1}(\dot{B}^{-\sigma_{1}-1}_{q,\infty})}\lesssim\bar{X}^{2}_{p,q}(T)+\bar{X}_{p,q}(T)\|\nabla A\|^{\ell}_{\tilde{L}^\infty_{T}(\dot{B}^{-\sigma_{1}-1}_{q,\infty})}.
\end{eqnarray}
We finally estimate $G_{6}\triangleq e^{\sqrt {c_{0}t}\Lambda_{1}} g_6$. It only need to deal with $\frac{\rho^{*}\check{\kappa}}{2}e^{\sqrt {c_{0}t}\Lambda_{1}}\nabla\big(|\nabla a|^{2}\big)$. Here, we utilize (\ref{mark1}) again that
 \begin{multline}
\Big\|\frac{\rho^{*}\check{\kappa}}{2}e^{\sqrt {c_{0}t}\Lambda_{1}}\nabla\big(|\nabla a|^{2}\big)\Big\|^{\ell}_{\tilde{L}_{T}^{1}(\dot{B}^{-\sigma_{1}-1}_{q,\infty})}\lesssim\|\nabla A\|_{\tilde{L}^2_{T}(\dot{B}^{1-\sigma_{1}+\frac{d}{p}-\frac{d}{q}}_{p,\infty})}
\|\nabla A\|_{\tilde{L}^2_{T}(\dot{B}^{\frac{d}{p}-1}_{p,1})}\\
\lesssim\big(\|A\|^{h}_{\tilde{L}^2_{T}(\dot{B}^{\frac{d}{p}+1}_{p,1})}+\|A\|^{\ell}_{\tilde{L}^2_{T}(\dot{B}^{-\sigma_{1}+1}_{q,\infty})}\big)
\big(\|A\|^{h}_{\tilde{L}^2_{T}(\dot{B}^{\frac{d}{p}+1}_{p,1})}+
\|A\|^{\ell}_{\tilde{L}^2_{T}(\dot{B}^{\frac{d}{q}-1}_{q,\infty})}\big),
\end{multline}
which leads to
 \begin{multline}
\Big\|\frac{\rho^{*}\check{\kappa}}{2}e^{\sqrt {c_{0}t}\Lambda_{1}}\nabla\big(|\nabla a|^{2}\big)\Big\|^{\ell}_{\tilde{L}_{T}^{1}(\dot{B}^{-\sigma_{1}-1}_{q,\infty})} \lesssim\bar{X}^{2}_{p,q}(T)+\bar{X}_{p,q}(T)\|\nabla A\|^{\ell}_{\tilde{L}^\infty_{T}(\dot{B}^{-\sigma_{1}-1}_{q,\infty})\cap\tilde{L}^1_{T}(\dot{B}^{-\sigma_{1}+1}_{q,\infty})}.
\end{multline}
Bounding the composite terms in $G_{6}$ can enjoy the similar steps with the help of (\ref{mark1}) and Proposition \ref{prop5.2}. Consequently, we get
 \begin{multline}\label{R-E3.51}
\|G_{6}(t)\|^{\ell}_{\tilde{L}_{T}^{1}(\dot{B}^{-\sigma_{1}-1}_{q,\infty})}\lesssim(1+\bar{X}_{p,q}(T))\bar{X}^{2}_{p,q}(T)\\
+(1+\bar{X}_{p,q}(T))\bar{X}_{p,q}(T)\|\nabla A\|^{\ell}_{\tilde{L}^\infty_{T}(\dot{B}^{-\sigma_{1}-1}_{q,\infty})\cap\tilde{L}^1_{T}(\dot{B}^{-\sigma_{1}+1}_{q,\infty})}.
\end{multline}

Combining with \eqref{Gh}, \eqref{R-E3.31}, \eqref{R-E3.37}, \eqref{R-E3.388}, \eqref{pressure}, \eqref{R-EKort} and \eqref{R-E3.51}, we achieve \eqref{energy esti}. Therefore, the proof of Lemma \ref{tol} is eventually finished.
\end{proof}

Furthermore, it follows from Theorem \ref{thm2.1} that
$$\bar{X}_{p,q}(T)\lesssim \|(\nabla a_{0}, m_{0})\|_{\dot{B}^{\frac{d}{p}-1,\frac{d}{q}-3}_{(p,1)(q,\infty)}}\ll1,$$
consequently, \eqref{energy esti} implies that the claim (\ref{evoo}) is true as $\|f\|_{\widetilde{L}^{\infty}_{T}(\dot{B}^{s}_{p,\infty})}\approx\|f\|_{L^{\infty}_{T}(\dot{B}^{s}_{p,\infty})}$.

\section{The proof of Theorem \ref{thm2.3}}\setcounter{equation}{0}\label{sec:4}
The last section is devoted to the proof of Theorem \ref{thm2.3}. It follows that
\begin{eqnarray}\label{decom}\|\Lambda^{l}(a,m)\|_{L^r}\leq\|\Lambda^{l}(a,m)\|^{\ell}_{\dot{B}^{0}_{r,1}}+\|\Lambda^{l}(a,m)\|^{h}_{\dot{B}^{0}_{r,1}}.\end{eqnarray}
We first show solutions decay polynomially in the low frequencies as fast as heat kernel. In fact, by utilizing embedding $\dot{B}^{\frac{d}{q}-\frac{d}{r}}_{q,1}\hookrightarrow\dot{B}^{0}_{r,1}$ and taking $\sigma=-\sigma_{1}$ in Lemma \ref{lem5.2}, one can deduce that
\begin{eqnarray}\label{double}
\|\Lambda^{l}a\|^{\ell}_{\dot{B}^{0}_{r,1}}
\lesssim\|\Lambda^{l}a\|^{\ell}_{\dot{B}^{\frac{d}{q}-\frac{d}{r}}_{q,1}}
\lesssim \|\Lambda^{l+\tilde{\sigma}_{1}}a\|^{\ell}_{\dot{B}^{-\sigma_{1}}_{q,1}}
\lesssim t^{-\frac{l+\tilde{\sigma}_{1}}{2}}\|e^{\sqrt {c_{0}t} \Lambda_{1}}a\|^{\ell}_{\dot{B}^{-\sigma_{1}}_{q,\infty}}
\end{eqnarray}
for $-\tilde{\sigma}_{1}<l$, where $\tilde{\sigma}_{1}=\sigma_{1}-\frac{d}{r}+\frac{d}{q}$. It is also obvious from Lemma \ref{lem5.1} that for $-\tilde{\sigma}_{1}<l$
\begin{eqnarray}\label{kiko}
\|\Lambda^{l}a\|^{\ell}_{\dot{B}^{0}_{r,1}}
\lesssim\|a\|^{\ell}_{\dot{B}^{-\sigma_{1}}_{q,\infty}}
\lesssim\|e^{\sqrt {c_{0}t} \Lambda_{1}}a\|^{\ell}_{\dot{B}^{-\sigma_{1}}_{q,\infty}}.
\end{eqnarray}
Let $t_0>0$ be some transient time. Choosing $t_0$ small enough such that $t_0<\frac{1}{2}$ yields $\langle t \rangle\sim \langle t-t_{0}\rangle$. Therefore, the combination of \eqref{double}-\eqref{kiko} implies that
\begin{eqnarray}\label{33}
\|\Lambda^{l}a\|^{\ell}_{\dot{B}^{0}_{r,1}}
\lesssim \langle t-t_{0}\rangle^{-\frac{\tilde{\sigma}_{1}}{2}-\frac{l}{2}}\|e^{\sqrt {c_{0}t} \Lambda_{1}}a\|^{\ell}_{\dot{B}^{-\sigma_{1}}_{q,\infty}}.
\end{eqnarray}
Bounding the large-time behavior of momentum $m$ is totally similar:
\begin{eqnarray}\label{lh00}
\|\Lambda^{l}m\|^{\ell}_{\dot{B}^{0}_{r,1}}
\lesssim \langle t-t_{0}\rangle^{-\frac{\tilde{\sigma}_{1}}{2}-\frac{1}{2}-\frac{l}{2}}\|e^{\sqrt {c_{0}t} \Lambda_{1}}m\|^{\ell}_{\dot{B}^{-\sigma_{1}-1}_{q,\infty}}
\end{eqnarray}
for $-\tilde{\sigma}_{1}-1<l$.

On the other hand, it is shown that the decay of high frequencies of solutions is actually exponential in large time. Indeed, by taking advantage of Sobolev embedding $\dot{B}^{\frac{d}{p}-\frac{d}{r}}_{p,1}\hookrightarrow\dot{B}^{0}_{r,1}$, we infer that
\begin{eqnarray*}
\|\Lambda^{l}a\|^{h}_{\dot{B}^{0}_{r,1}}
\lesssim\|\Lambda^{l}a\|^{h}_{\dot{B}^{\frac{d}{p}-\frac{d}{r}}_{p,1}}
\lesssim\|\Lambda^{\alpha}a\|^{h}_{\dot{B}^{\frac{d}{p}}_{p,1}}
\end{eqnarray*}
for $\alpha\geq l-\frac{d}{r}$.

Furthermore, by taking $\sigma=\frac{d}{p}$ in Lemma \ref{lem5.2} for high frequencies and using the embedding $l^1\hookrightarrow l^\infty$, we immediately get
\begin{eqnarray}\label{dh}
\|\Lambda^{l}a\|^{h}_{\dot{B}^{0}_{r,1}}
\lesssim t^{-\frac{\alpha}{2}}e^{-c\sqrt{t}}\|e^{\sqrt {c_{0}t} \Lambda_{1}}a\|^{h}_{\dot{B}^{\frac{d}{p}}_{p,1}}
\end{eqnarray}
for any $\alpha\geq0$. By performing similar calculations, one has $$\|\Lambda^{l}m\|^{h}_{\dot{B}^{0}_{r,1}}\lesssim t^{-\frac{\alpha}{2}}e^{-c\sqrt{t}}\|e^{\sqrt {c_{0}t} \Lambda_{1}}m\|^{h}_{\dot{B}^{\frac{d}{p}-1}_{p,1}},$$ provided that $\alpha\geq\max\{l-\frac{d}{r}+1,0\}$.

The last step is to estimate the norm $\|\cdot\|_{L^r}$ of solution and its higher order derivatives, which is actually bounded for all $t\geq t_{0}$ by Gevrey analyticity.
Specifically speaking, keep in mind that (\ref{dh}), it holds that
 \begin{eqnarray*}
 \|\Lambda^{l}a\|^{h}_{\dot{B}^{0}_{r,1}}
\lesssim t^{-\frac{\alpha}{2}}e^{-c\sqrt{t}}\|e^{\sqrt {c_{0}t} \Lambda_{1}}a\|^{h}_{\dot{B}^{\frac{d}{p}}_{p,1}}
\end{eqnarray*}
with $\alpha\geq\max\{l-\frac{d}{r},0\}$ and $t>0$.

Notice that the function $t^{-\frac{\gamma}{2}}e^{-c\sqrt{t}}(\gamma\geq0)$ is monotonically decreasing for $t\geq t_{0}$, we are led to
 \begin{eqnarray}\label{guha}
 \|\Lambda^{l}a\|^{h}_{\dot{B}^{0}_{r,1}}
\lesssim t_{0}^{-\frac{\alpha}{2}}e^{-c\sqrt{t_{0}}}\|e^{\sqrt {c_{0}t} \Lambda_{1}}a\|^{h}_{\dot{B}^{\frac{d}{p}}_{p,1}},
\end{eqnarray}
furthermore,
 \begin{eqnarray}\label{guha}
\|\Lambda^{l}a\|^{h}_{\dot{B}^{0}_{r,1}}
 \leq C_{t_{0}} \langle t-t_{0}\rangle^{-\frac{\alpha}{2}}e^{-c\sqrt{t}}\|e^{\sqrt {c_{0}t} \Lambda_{1}}a\|^{h}_{\dot{B}^{\frac{d}{p}}_{p,1}},
\end{eqnarray}
where $C_{t_{0}}$ is some positive constant depending on $t_{0}$. Therefore, there exists some $\tilde{\alpha}>0$ large enough such that
\begin{eqnarray}\label{mlk}
 \|\Lambda^{l}a\|^{h}_{\dot{B}^{0}_{r,1}}
\leq C_{t_{0}} \langle t-t_{0}\rangle^{-\frac{\tilde{\alpha}}{2}}\|e^{\sqrt {c_{0}t} \Lambda_{1}}a\|^{h}_{\dot{B}^{\frac{d}{p}}_{p,1}}
\end{eqnarray}
for all $t\geq t_{0}$. Arguing similarly for the momentum $m$, we arrive at
\begin{eqnarray} \label{R-E4.10}\|\Lambda^{l}m\|^{h}_{\dot{B}^{0}_{r,1}}\leq C_{t_{0}} \langle t-t_{0}\rangle^{-\frac{\tilde{\alpha}}{2}}\|e^{\sqrt {c_{0}t} \Lambda_{1}}m\|^{h}_{\dot{B}^{\frac{d}{p}-1}_{p,1}}.\end{eqnarray}

Finally, by combining above estimates \eqref{33}-\eqref{lh00}, \eqref{mlk}-\eqref{R-E4.10} along with (\ref{evoo}) and Theorem \ref{thm2.1}, one can conclude \eqref{bound decay1}-\eqref{bound decay2} immediately. Hence, the proof of  Theorem \ref{thm2.3} is complete.  $\square$

\bigskip\noindent
{\bf Acknowledgments:}\ \
The second author (J. Xu) is partially supported by the National Natural Science Foundation of China (11871274, 12031006).

Zihao Song

Department of Mathematics,

Nanjing University of Aeronautics and Astronautics,

Nanjing 211106, People's Republic of China.

E-mail: szh1995@nuaa.edu.cn

\vskip 0.3cm

Jiang Xu

Department of Mathematics,

Nanjing University of Aeronautics and Astronautics,

Nanjing 211106, People's Republic of China.

E-mail: jiangxu\underline{\ \ }79@nuaa.edu.cn; jiangxu\underline{ }79math@yahoo.com

\end{document}